\documentclass[12pt,reqno]{amsart}
\usepackage{mathtools}
\usepackage{latexsym,bbm}
\usepackage{mathrsfs}

\mathtoolsset{showonlyrefs}
\usepackage[hmarginratio={1:1},scale=0.8]{geometry}
\usepackage{enumerate}
\usepackage{amssymb}
\usepackage{cite,color}
\usepackage{bm} \numberwithin{equation}{section}

\newcommand{\mycomment}[1]{}

\DeclareMathOperator{\Span}{span}

\DeclareMathOperator{\Int}{int}

\newcommand{\ko}{\mathcal K_o^n}
\newcommand{\kn}{\mathcal K^n}

\newcommand{\ke}{\mathcal K_e^n}
\newcommand{\rn}{\mathbb R^n}

\newcommand{\sn}{ {S^{n-1}}}
\newcommand{\wt}{\widetilde}

\newtheorem{lemma}{Lemma}[section]
\newtheorem{theorem}[lemma]{Theorem}

\newtheorem{coro}[lemma]{Corollary}

\newtheorem{prop}[lemma]{Proposition}
\newtheorem{remark}[lemma]{Remark}
\title{The $L_p$ chord Minkowski problem in a critical interval}
\author[L. Guo]{Lujun Guo}
\address{College of Mathematics and Information Science, Henan Normal University, Henan 453007, China}
\email{lujunguo03018@163.com}

\author[D. Xi]{Dongmeng Xi}
\address{Department of Mathematics,
Shanghai University,
Shanghai 200444, China}
\email{xi\_dongmeng@shu.edu.cn}

\author[Y. Zhao]{Yiming Zhao}
\address{Department of Mathematics,
Syracuse University,
Syracuse, NY 13244, USA}
\email{yzhao197@syr.edu}

\begin{document}
\begin{abstract}
Chord measures and $L_p$ chord measures were recently introduced by Lutwak-Xi-Yang-Zhang by establishing a variational formula regarding a family of fundamental integral geometric invariants called chord integrals. Prescribing the $L_p$ chord measure is known as the $L_p$ chord Minkowski problem, which includes the $L_p$ Minkowski problem heavily studied in the past 2 decades as special cases. In the current work, we solve the $L_p$ chord Minkowski problem when $0\leq p<1$, without symmetry assumptions.
\end{abstract}

\maketitle


\section{Introduction}

Central to the theory of convex bodies are geometric invariants and measures associated with convex bodies. Geometric invariants and measures are usually investigated through \emph{isoperimetric inequalities} and \emph{Minkowski problems}. They are intimately connected. As an example, the celebrated \emph{Brunn-Minkowski inequality} reveals that the volume functional is log-concave in a certain sense and the classical isoperimetric inequality, as a direct consequence of it, reveals that ball is the geometric shape that minimizes surface area among convex bodies with fixed volume. The \emph{classical Minkowski problem} asks for the existence, uniqueness, and regularity of a convex body whose \emph{surface area measure} is equal to a pre-given spherical Borel measure. The two problems are closely connected since surface area measure can be viewed as the ``{derivative}'' of the volume functional. The classical Minkowski problem has motivated much of the study of fully nonlinear partial differential equation, as demonstrated by the works of Minkowski \cite{MR1511220}, Aleksandrov \cite{MR0001597}, Cheng-Yau \cite{MR0423267}, Pogorelov\cite{MR0478079}, and Caffarelli \cite{MR1005611, MR1038359,MR1038360} throughout the last century.

The volume functional is a special case of \emph{quermassintegrals} that includes surface area and mean width as two other more well-known invariants. Quermassintegrals are fundamental invariants in the classical Brunn-Minkowski theory. Depending on parametrization, their ``derivatives'' include the \emph{area measures} introduced by Aleksandrov, Fenchel, and Jessen in the 1930s, as well as the \emph{curvature measures} introduced by Federer in the late 1950s.
With sufficient regularity assumptions on the convex body, area measures and curvature measures involve elementary symmetric functions of principal curvatures and radii of curvature. This makes them much more complicated than the surface area measure studied in the classical Minkowski problem. Minkowski problems for area measures and curvature measures include the \emph{Christoffel problem} (for the area measure $S_1$) and the long-standing \emph{Christoffel-Minkowski problem} (for the area measure $S_{n-2}$).\footnote{As a comparison, the classical Minkowski problem studies the surface area measure which is also known as the area measure $S_{n-1}$.} See, for example, Guan-Guan \cite{MR1933079},  Guan-Li-Li \cite{MR2954620}, Guan-Ma \cite{MR1961338}, Guan-Ma-Zhou \cite{MR2237290}.

In the 1970s, Lutwak introduced the dual Brunn-Minkowski theory. Compared to the classical theory which focuses more on projections and boundary shapes of convex bodies, the dual Brunn-Minkowski theory focuses more on intersections and interior properties of convex bodies. This explains the crucial role that the dual theory played in the solution of the well-known and the then long-standing \emph{Busemann-Petty problem} in the 1990s. See, for example, \cite{MR1298719, MR1689343, MR963487, MR1689339}.
The counterparts for the quermassintegrals in the dual theory are the \emph{dual quermassintegrals}. See Section \ref{preliminary invariants}. However, it was not until the groundbreaking work \cite{MR3573332} of Huang-Lutwak-Yang-Zhang (Huang-LYZ) that the geometric measures associated with dual quermassintegrals were revealed. This led to \emph{dual curvature measures} dual to Federer's curvature measures. The Minkowski problem for dual curvature measures, now known as the dual Minkowski problem, has been the focus in convex geometry and fully nonlinear elliptic PDEs for the last couple of years and has already led to a number of papers in a short period. See, for example, B\"{o}r\"{o}czky-Henk-Pollehn \cite{MR3825606}, Chen-Chen-Li \cite{MR4259871}, Chen-Huang-Zhao \cite{MR3953117}, Chen-Li \cite{MR3818073}, Gardner-Hug-Weil-Xing-Ye \cite{MR3882970}, Henk-Pollehn \cite{MR3725875}, Li-Sheng-Wang \cite{MR4055992}, Liu-Lu \cite{MR4127893}, Zhao \cite{MR3880233}. It is important to note that the list is by no means exhaustive.

Unlike quermassintegrals, dual quermassintegrals, which depend on lower dimensional \emph{central} sectional areas, are \emph{not} translation invariant. Integrating dual quermassintegrals of a convex body over all its translated copies (that contain the origin) leads to a basic geometric invariant in integral geometry, known as \emph{chord integral}. Chord integrals are naturally translation invariant. From an analysis point of view, chord integrals are Riesz potentials of characteristic functions of convex bodies. For isoperimetric problems involving chord integrals, readers should refer to Kn\"{u}pfer-Muratov \cite{MR3055587, MR3272365}, Figalli-Fusco-Maggi-Millot-Morini \cite{MR3322379}, Haddad-Ludwig \cite{https://doi.org/10.48550/arxiv.2209.10540} and the references cited therein.

Recently, the ``derivative'' of chord integrals, called \emph{chord measures}, was obtained in Xi-LYZ \cite{XLYZ}. The Minkowski problem for chord measures was posed and studied in the same paper. It is called the chord Minkowski problem. The chord Minkowski problem includes the classical Minkowski problem and the previously mentioned long-standing Christoffel-Minkowski problem---the latter as a critical limiting case. The $L_p$ extensions of chord measures and the chord Minkowski problem are natural and present many interesting and challenging problems. More details on this will follow. The $L_0$ chord Minkowski problem, in particular, is also known as the \emph{chord log-Minkowski problem} as it contains the unsolved logarithmic Minkowski problem (see, for example, \cite{MR3037788}) as a special case.

Xi-LYZ \cite{XLYZ} solved \emph{completely} the chord Minkowski problem (corresponding to $p=1$) except for the limiting Christoffel-Minkowski problem case and they also demonstrated a sufficient condition for the \emph{$o$-symmetric} case of the chord log-Minkowski problem. Xi-Yang-Zhang-Zhao \cite{XYZZ} solved the $L_p$ chord Minkowski problem for $p>1$ as well as the \emph{$o$-symmetric} case of $0<p<1$. Origin symmetry in the case of $0\leq p<1$ plays an important role in obtaining \emph{a-priori} $C^0$ bounds---even more so in the critical $p=0$ case.

 The purpose of the current paper is to show that the symmetric restriction in both works can be dropped via an approximation scheme from the polytopal case.

 Let $K$ be a convex body in $\rn$. For each $q\geq 0$, the $q$-th chord (power) integral of $K$, denoted by $I_q(K)$, is given by
 \begin{equation}
 	I_q(K) = \int_{\mathscr{L}^n}
|K\cap \ell|^q\, d\ell,
 \end{equation}
 where $|K\cap \ell|$ is the length of the chord $K\cap \ell$ and the integration is with respect to Haar measure on the affine Grassmannian $\mathscr{L}^n$. Chord integrals contain volume and surface area as important special cases:
 \begin{equation}
 	I_0(K)=\frac{\omega_{n-1}}{n\omega_n}S(K),\quad I_1(K)=V(K), \quad I_{n+1}(K)= \frac{n+1}{\omega_n}V(K)^2,
 \end{equation}
 where $\omega_n$ is the volume of the unit ball in $\rn$. In particular, the chord integral $I_q$ for $q\in (0,1)$ can be seen as an interpolation between volume (or, the quermassintegral $W_0$) and surface area (or, the quermassintegral $W_1$). Chord integrals also take the form of Riesz potential, see \eqref{eq 10221}.

 Xi-LYZ \cite{XLYZ} demonstrated that for each $q\geq 0$, the ``derivative'' of the chord integral $I_q(K)$ uniquely defines the chord measure $F_q(K,\cdot)$ on $\sn$:
 \begin{equation}
 \label{eq 10222}
 	\frac{d}{dt}\bigg|_{t=0^+} I_q(K+tL)=\int_\sn h_L(v)dF_q(K,v),
 \end{equation}
 for each pair of convex bodies $K$ and $L$. Here $h_L$ is the support function of $L$. A precise definition of the chord measure $F_q$ can be found in Section \ref{preliminary measure}. It is important to note that the $q=0,1$ cases of \eqref{eq 10222} are classic and in such cases, the chord measure $F_q(K,\cdot)$ recovers surface area measure ($q=1$) and the area measure $S_{n-2}(K,\cdot )$ ($q=0$). In this way, the chord measure $F_q(K,\cdot)$ interpolates between surface area measure and the $(n-2)$-th order area measure.
\vskip 5pt
\noindent \textbf{The chord Minkowski problem}. Given a finite Borel measure $\mu$ on $\sn$, what are the necessary and sufficient conditions on $\mu$ so that there exists a convex body $K$ such that $F_q(K,\cdot)=\mu$?
\vskip 5pt
The chord Minkowski problem recovers the classical Minkowski problem (when $q=1$) and the \emph{long-standing} Christoffel-Minkowski problem (when $q=0$). The chord Minkowski problem for $q>0$ was completely solved in \cite{XLYZ}.

In the past three decades, many classical concepts and results in the theory of convex bodies have been extended to their $L_p$ counterparts. This was initiated by two landmark papers \cite{MR1231704, MR1378681} by Lutwak in the early 1990s where he defined the \emph{$L_p$ surface area measure} fundamental in the now fruitful $L_p$ Brunn-Minkowski theory central in modern convex geometric analysis. It is crucial to point out that such extension is highly nontrivial and often requires new techniques. See, for example, \cite{MR2254308, MR2132298, MR1316557,MR2067123,MR3366854,MR3479715,MR3356071,MR2123199,MR3415694,MR3148545,MR2652209,MR2680490,MR1901250,MR2019226,MR2729006,MR2530600,MR1987375,MR2927377} for a (not even close to exhaustive) list of works in the $L_p$ Brunn-Minkowski theory. In particular, the theory becomes \emph{significantly} harder when $p<1$. These include the critical centro-affine case $p=-n$ and the logarithmic case $p=0$. Isoperimetric inequalities and Minkowski problems in neither case have been fully addressed. In particular, the \emph{log Minkowski problem} (for the \emph{cone volume measure}) has not yet been fully solved. See, for example, Bianchi-B\"or\"oczky-Colesanti-Yang \cite{MR3872853}, Chou-Wang \cite{MR2254308}, Guang-Li-Wang \cite{GLW2022}, Zhu \cite{MR3356071, MR3228445} among many other works. In fact, the $p=0$ case harbors the \emph{log Brunn-Minkowski conjecture} (see, for example, B\"or\"oczky-LYZ \cite{MR2964630})---arguably the most crucial conjecture in convex geometric analysis in the past decade. The log Brunn-Minkowski conjecture has been verified in dimension 2 and in various special classes of convex bodies. See, for example, Chen-Huang-Li-Liu \cite{MR4088419}, Colesanti-Livshyts-Marsiglietti \cite{MR3653949},  Kolesnikov-Livshyts \cite{MR4485961}, Kolesnikov-Milman \cite{MR4438690},  Milman \cite{MJEMS}, Putterman \cite{MR4220744}, Saroglou \cite{MR3370038}. If proven correct, it is much stronger than the classical Brunn-Minkowski inequality.

Motivated by this success, the $L_p$ chord measure was introduced in \cite{XLYZ}. For each $p\in \mathbb{R}$, $q>0$, and convex body $K$ containing the origin in its interior, the $(p,q)$-th chord measure, denoted by $F_{p,q}(K,\cdot)$, is a finite Borel measure on $\sn$ given by
\begin{equation}
	dF_{p,q}(K,\cdot) = h_K^{1-p} dF_q(K,\cdot).
\end{equation}
For $p\leq 1$, since the exponent $1-p$ is nonnegative, the above definition naturally extends to all $K\in \kn$ as long as $o\in K$. We point out that when $q=1$, since the chord measure $F_q$ becomes the surface area measure, the $(p,1)$-th chord measure becomes nothing but the family of $L_p$ surface area measure in the $L_p$ Brunn-Minkowski theory.
\vskip 5pt
 \noindent \textbf{The $L_p$ chord Minkowski problem}. Given $p\in \mathbb{R}$, $q>0$, and a finite Borel measure $\mu$ on $\sn$, what are the necessary and sufficient conditions on $\mu$ so that there exists a convex body $K$ containing the origin (as an interior point if $p>1$) such that $F_{p,q}(K,\cdot)=\mu$?
\vskip 5pt
 When the given measure $\mu$ has a nonnegative density $f$, the $L_p$ chord Minkowski problem reduces to solving the following Monge-Amp\`{e}re type equation on $\sn$:
 \begin{equation}
 \label{eq 10224}
 	h_K^{1-p}{\wt V_{q-1}(K,\nabla h_K)}\det\big(\nabla^2_{\sn}h_K + h_K\delta_{ij}) = {f}.
 \end{equation}
 Here $\nabla^2_\sn h_K$ is the Hessian of $h_K$ on the unit sphere with respect to the standard metric, and $\nabla h_K$ is the Euclidean gradient of $h_K$ that is connected to the spherical gradient $\nabla_\sn h_K$ in the following way:
 \[ \nabla h_K(v) = \nabla_{\sn} h_K(v) + h_K(v)v.\]
We remark at this point that when $q=1$, the $L_p$ chord Minkowski problem reduces to the $L_p$ Minkowski problem.

In \cite{XYZZ}, it was shown that if $p\in (0,1)$, $q>0$, and the given measure $\mu$ is an even measure, then the $L_p$ chord Minkowski problem has an $o$-symmetric solution. The origin-symmetry assumption is heavily utilized there so that \emph{a-priori} bounds can be achieved. If the origin-symmetry assumption is dropped, then the situation is vastly different. In fact, the maximization problem used in \cite{XYZZ} (for the sake of variational approach) is no longer applicable in the general case. Similar to the $L_p$ Minkowski problem for $p<1$, a min-max problem has to be considered---in another word, we are instead searching for a saddle point. The first of our main results is the following:

\begin{theorem}
\label{thm intro 1}
	Let $0<p<1$, $q>0$, and $\mu$ be a finite Borel measure on $\sn$ not concentrated in any closed hemisphere. Then, there exists $K\in \mathcal{K}^n$ with $o\in K$ such that
	\begin{equation}
		F_{p,q}(K,\cdot) = \mu.
	\end{equation}
Moreover, if $\mu$ is a finite discrete measure, then $K$ is a polytope that contains the origin as an interior point.
\end{theorem}

To prove Theorem \ref{thm intro 1}, we first establish the case when $\mu$ is discrete. This is contained in Theorem \ref{thm discrete'}. The polytopal solutions are then used to obtain the general solution via an approximation scheme (Theorem \ref{thm main thm p neq 0}). In particular, Theorem \ref{thm intro 1} contains the solution to the $L_p$ Minkowski problem when $0<p<1$ previously obtained in Zhu \cite{MR3352764} and Chen-Li-Zhu \cite{MR3680945}.

When $p=0$, (up to a constant) the $L_0$ chord measure is also known as the \emph{cone chord measure} $G_q$:
\begin{equation}
	G_q(K,\cdot) = \frac{1}{n+q-1} F_{0,q}(K,\cdot).
\end{equation}
The special normalization is so that $G_q(K,\sn) = I_q(K)$. See Section \ref{preliminary measure} for details. We remark that the cone chord measure $G_1(K,\cdot)$ is equal to the cone volume measure $V_K$ sitting at the center of the aforementioned log-Brunn-Minkowski conjecture. Recall that the Minkowski problem for cone volume measure is known as the log Minkowski problem.
For this reason, we also refer to the $L_0$ chord Minkowski problem as the chord log-Minkowski problem.

It turns out that the chord log-Minkowski problem is connected to a subspace mass inequality.
Let $1<q<n+1$. We say that a given finite Borel measure $\mu$ satisfies the \emph{subspace mass inequality} if
\begin{equation}
\label{eq 1068}
	\frac{\mu(\xi_i\cap \sn)}{|\mu|} <\frac{i+\min\{i, q-1\}}{n+q-1},
\end{equation}
for each $i$ dimensional subspace $\xi_i\subset \rn$ and each $i=1, \dots, n-1$.

It was shown in \cite{XLYZ} that with the additional assumption that $\mu$ is even, \eqref{eq 1068} is sufficient to guarantee an $o$-symmetric solution $K\in \ko$ such that $\mu = G_q(K,\cdot)$. We show in the current work that the symmetric assumption can be removed. We remark at this point that as \eqref{eq 1068} demonstrates, the chord log-Minkowski problem with general $\mu$ is much more complicated than its special case when $\mu$ is absolutely continuous (\emph{i.e.,} equation \eqref{eq 10224}). Indeed, if $\mu$ is absolutely continuous, then its mass in any proper subspace is 0 and therefore the subspace mass inequality \eqref{eq 1068} is trivially satisfied.

To solve the chord log-Minkowski problem, we first prove the polytopal case when the given normal vectors are in \emph{general position}. Polytopes possessing this special feature have the additional property that if they blow up (collapse, resp.), then they have to blow up (collapse, resp.) in a uniform fashion. This will make it easier to obtain uniform \emph{a-priori} bounds. Vectors in general position and polytopes with normals in general position will be discussed in Section \ref{sec general position}. Using this, we will show

\begin{theorem}
\label{thm intro discrete}
	Let $q>0$, and $\mu$ be a discrete measure on $\rn$ whose support set is not contained in any closed hemisphere and is in general position in dimension $n$. Then there exists a polytope $P$ containing the origin in its interior such that
	\begin{equation}
		G_q(P,\cdot)=\mu.
	\end{equation}
\end{theorem}
Theorem \ref{thm intro discrete} is implied by Theorem \ref{thm discrete} and the homogeneity of $G_q(P,\cdot)$ in $P$.

Section \ref{sec general p=0} is devoted to using Theorem \ref{thm intro discrete} and an approximation scheme to show:

\begin{theorem}
\label{thm intro 2}
	Let $1<q<n+1$. If $\mu$ is a finite Borel measure on $\sn$ that satisfies \eqref{eq 1068}, then there exists a convex body $K\in \mathcal{K}^n$ with $o\in K$ such that
	\begin{equation}
		G_q(K,\cdot)=\mu.
	\end{equation}
\end{theorem}
We remark that Theorem \ref{thm intro discrete} and Theorem \ref{thm intro 2} extend the previously obtained results on the log Minkowski problem in Zhu \cite{MR3228445} and Chen-Li-Zhu \cite{MR3896091}.

\section{Preliminaries}

In this section, we gather notations and results needed in subsequent sections.

\subsection{Basics of convex bodies}

The central objects in study in convex geometry are convex bodies which are nothing but compact convex sets in $\rn$ with non-empty interiors. It is important to note that we require no additional regularity other than convexity of the set. We will write $\kn$ for the set of all convex bodies in $\rn$. The symbols $\ko$ will be used for the subclass of $\kn$ that contains convex bodies that have the origin in their interiors and $\ke$ will be used for the subclass of $o$-symmetric convex bodies. We write $\omega_n$ for the volume of the unit ball in $\rn$. We will also use the notation $|\mu|$ for the total mass of a measure $\mu$. For each positive integer $N$, the symbol $\mathbb{R}^N_+$ stands for the collection of points in $\mathbb{R}^N$ with positive coordinates; that is, 
\begin{equation*}
	\mathbb{R}_+^N=\{(x_1,\dots, x_N): x_i>0\}.
\end{equation*}

Readers should consult the classical volume \cite{MR3155183} by Schneider for details of the results covered in this section.

A compact convex set $K$ is uniquely determined by its support function $h_K: \sn\rightarrow \mathbb{R}$ given by
\begin{equation*}
	h_K(v) = \max_{x\in K} x\cdot v.
\end{equation*}
It is worth noting that the support function can be trivially extended to $\rn$ as a $1$-homogeneous function and it is convex.

Let $K\in \kn$ and $x\in \rn$. The radial function of $K$ with respect to $x$, denoted by $\rho_{K,x}: \sn \rightarrow \mathbb{R}$ can be written as
\begin{equation}
	\rho_{K,x} (u)=\max \{t: tu+x\in K\}.
\end{equation}
It is simple to see that when $x\in \Int K$, we have that $\rho_{K,x}$ is a positive continuous function on $\sn$. For simplicity, we write $\rho_K=\rho_{K,o}$.

We will use $\nu_K: \partial K\rightarrow \sn$ to denote the Gauss map of $K$. In particular, the convexity of $K$ implies that $\nu_K$ is almost everywhere defined on $\partial K$.

Since all support functions have to be convex, it is obvious that not all functions on $\sn$ are support functions of convex bodies. However, the so-called \emph{Wulff shape} or \emph{Aleksandrov body} connects continuous functions defined on subsets of $\sn$ to convex bodies. In particular, let $\Omega\subset \sn$ be a subset that is not entirely contained in any closed hemisphere and $f:\Omega\rightarrow [0,\infty)$ be a continuous function. The Wulff shape $[f,\Omega]$ is defined to be
\begin{equation}
	[f,\Omega] = \{x\in \rn: x\cdot v\leq f(v), \forall v\in \Omega\}.
\end{equation}
It is clear that $[f,\Omega]$ is convex and compact. Moreover when $f>0$, the Wulff shape $[f,\Omega]$ contains the origin as an interior point. For simplicity, when the context is clear, we shall write $[f]$ without explicitly mentioning $\Omega$. It is simple to see that
\begin{equation}
	\label{eq 6151}h_{[f]}\leq f.
\end{equation}
It is important that the above inequality may very well be strict for many $f$. A critical observation regarding Wulff shape is that for almost all $x\in \partial [f]$, the normal vector $\nu_{[f]}(x)\in \Omega$.

Let $K_n$ be a sequence of compact convex sets in $\rn$. We say that $K_n$ converges to $K$ in \emph{Hausdorff metric} if $\|h_{K_n}-h_K\|_{\infty}\rightarrow0$ as $n\rightarrow \infty$. We shall use frequently the fact that if $f_i\in C(\Omega)$ converges to $f\in C(\Omega)$ uniformly, then $[f_i]\rightarrow [f]$ in Hausdorff metric.

When $\Omega=\{v_1,\dots, v_N\}$ is a finite set not contained in any closed hemisphere, we will slightly abuse our notation and for each $z=(z_1,\dots, z_N)\in \mathbb{R}^N$, write
\begin{equation}
	[z,\Omega] = \{x\in \rn: x\cdot v_i\leq z_i, \quad i=1,\dots, N\}.
\end{equation}
When the context is clear, we shall write $[z]$ for simplicity.
We will write $\mathcal{P}(v_1,\dots, v_N)$ for the collection of convex bodies generated in this fashion. Specifically, the set $\mathcal{P}(v_1,\dots, v_N)$ contains all polytopes in $\rn$ whose normals to facets are contained in $\{v_1,\dots, v_N\}$.

A special collection of polytopes are those whose facet normals are in \emph{general position} in dimension $n$. We say $v_1,\dots, v_N$ are in general position in dimension $n$ if for any $n$-tuple $1\leq i_1<i_2<\dots<i_n\leq N$, the vectors $v_{i_1}, \dots, v_{i_n}$ are linearly independent. In Section \ref{sec general position}, we will show that for polytopes whose normals are in general position, if they grow in size, then they have to grow uniformly.
\subsection{Invariants in integral geometry} \label{preliminary invariants}In this subsection, we gather notions from integral geometry. Readers are referred to the books \cite{MR2162874} by Santal\`{o} and \cite{MR1336595}  by Ren.

In the classical Brunn-Minkowski theory of convex bodies, \emph{quermassintegrals} $W_0, W_1,\dots, W_n$ are fundamental geometric invariants that include volume, surface area, and mean width as important special cases. They arise in many different ways. One way to see them is as coefficients of the \emph{Steiner formula} fundamental in the classical Brunn-Minkowski theory (see Section 4.2 in \cite{MR3155183}). It also naturally arises from an integral geometry point of view. The quermassintegrals $W_{n-i}$ can be defined as
\begin{equation}
	W_{n-i}(K) = \frac{\omega_n}{\omega_i} \int_{\xi \in G_{n,i}} \mathcal{H}^{i}(K|{\xi})d\xi,
\end{equation}
where $G_{n,i}$ is the Grassmannian manifold containing all $i$ dimensional subspaces of $\rn$, the set $K|\xi$ is the image of the orthogonal projection of $K$ onto $\xi$, and the integration is with respect to the Haar measure in $G_{n,i}$. Quermassintegrals satisfy the  \emph{fundamental kinematic formula}; see (4.54) in \cite{MR3155183}. With sufficient regularity assumptions on the boundary of the convex body, quermassintegrals are integrals of elementary symmetric polynomials of principal curvatures of the body.

While quermassintegrals are heavily connected to boundary shape and orthogonal projection areas of convex bodies, \emph{dual quermassintegrals} fundamental in the dual Brunn-Minkowski theory are related to interior properties and central sectional areas of convex bodies. They arise naturally as coefficients of the \emph{dual Steiner formula} (see Section 9.3 in \cite{MR3155183}). From an integral geometric point of view, for each $K\in \ko$, the dual quermassintegrals of $K$ can be defined as
\begin{equation}
	\widetilde{W}_{n-i}(K) = \frac{\omega_n}{\omega_i} \int_{\xi \in G_{n,i}} \mathcal{H}^{i}(K\cap{\xi})d\xi.
\end{equation}
It was shown in Zhang \cite{MR1443203} that the dual quermassintegrals enjoy a kinematic formula dual to the fundamental kinematic formula. Using polar coordinates, it is not hard to show that dual quermassintegrals satisfy an integral representation via radial functions:
\begin{equation}
	\widetilde{W}_{n-i}(K)=\frac{1}{n}\int_{\sn}\rho_K^i(u)du,
\end{equation}
which allows an immediate extension from $\widetilde{W}_{n-i}(K)$ to $\widetilde{W}_{n-q}(K)$ for each $q\in \mathbb{R}$. It is apparent that, unlike quermassintegrals, dual quermassintegrals are \emph{not} translation invariant in $K$. Therefore, we may define for each $z\in K$ and $q\in \mathbb{R}$:
\begin{equation}
	\widetilde{W}_{n-q}(K,z) = \frac{1}{n}\int_{S_z^+}\rho_{K,z}^q(u)du,
\end{equation}
where $S_z^+=\{u\in \sn: \rho_{K,z}(u)>0\}$. Note that when $z\in \Int K$, we have $S_z^+=\sn$. For the sake of notational simplicity, we will write $\widetilde{V}_{q}(K,z)=\widetilde{W}_{n-q}(K,z)$.

The integrals of dual quermassintegrals with respect to $z\in K$ naturally give rise to translation invariant quantities. These are known as \emph{chord integrals} in integral geometry. More specifically, let $q\geq 0$ and $K\in \kn$, the $q$-th chord (power) integral of $K$ is given by
\begin{equation}
 	I_q(K) = \int_{\mathscr{L}^n}
|K\cap \ell|^q\, d\ell,
 \end{equation}
where $|K\cap \ell|$ is the length of the chord $K\cap \ell$ and the integration is with respect to Haar measure on   $\mathscr{L}^n$ which denotes the  affine Grassmannian of lines ($1$-dimensional affine subspaces). For $q>0$, the chord integral can be written as the integral of dual quermassintegrals in $z\in K$:
\begin{equation}
	I_q(K)=\frac{q}{\omega_n} \int_K \widetilde{V}_{q-1}(K,z)dz.
\end{equation}
In analysis, chord integral can be recognized as Riesz potential: for each $q>1$, we have
\begin{equation}
\label{eq 10221}
	I_q(K) = \frac{q(q-1)}{n\omega_n} \int_K\int_K \frac{1}{|x-z|^{n+1-q}}dxdz.
\end{equation}
Aside from translation invariance, we shall make frequent use of the fact that $I_q$ is homogeneous of degree $n+q-1$, \textit{i.e., } $I_q(tK)=t^{n+q-1}I_q(K)$ for $t>0$. For $q\geq0$, there is an obvious extension of $I_q$ to the set of all compact convex subsets of $\rn$ and $I_q$ is a continuous functional with respect to the Hausdorff metric. The proof of these facts can be found in, for example, \cite{XLYZ}.

\subsection{$L_p$ chord measures}
\label{preliminary measure}

 In the landmark paper \cite{XLYZ}, a new family of geometric measures in the setting of integral geometry, called \emph{chord measures}, was defined. Let $K\in \kn$ and $q>0$, the chord measure $F_q(K,\cdot)$ is a finite Borel measure on $\sn$ given by
 \begin{equation}
 	F_q(K,\eta)= \frac{2q}{\omega_n}\int_{\nu_K^{-1}(\eta)} \widetilde{V}_{q-1}(K,z)d\mathcal{H}^{n-1}(z), \qquad \text{for each Borel }\eta\subset \sn.
 \end{equation}
 If $K$ is a polytope, its chord measure becomes a discrete measure that is concentrated on the set of facet normals of $K$. On the other side, when $K$ is $C^{2,+}$, the chord measure $F_q(K,\cdot)$ is absolutely continuous with respect to the spherical Lebesgue measure:
 \begin{equation}
 	dF_q(K,v) = \frac{2q}{\omega_n} \widetilde{V}_{q-1}(K, \nabla h_K) \det (\nabla^2_{\sn} h_K+h_KI) dv.
 \end{equation}

 Chord measures naturally appear when one differentiates in a certain sense the chord integral $I_q$. Particularly,
 \begin{theorem}[Theorem 5.5 in \cite{XLYZ}]
	\label{theorem variational formula}
	Let $q>0$, and $\Omega$ be a compact subset of $\sn$ that is not contained in any closed hemisphere. Suppose that $g:\Omega\rightarrow \mathbb{R}$ is continuous and $h_t:\Omega\rightarrow (0,\infty)$ is a family of continuous functions given by
	\begin{equation}
		h_t = h_0+tg +o(t,\cdot),
	\end{equation}
	for each $t\in (-\delta,\delta)$ for some $\delta>0$. Here $o(t,\cdot)\in C(\Omega)$ and $o(t,\cdot)/v$ tends to $0$ uniformly on $\Omega$ as $t\rightarrow 0$. Let $K_t$ be the Wulff shape generated by $h_t$ and $K$ be the Wulff shape generated by $h_0$. Then,
	\begin{equation}
		\left.\frac{d}{dt}\right|_{t=0} I_q(K_t) = \int_{\Omega} g(v)dF_q(K,v).
	\end{equation}
\end{theorem}
\begin{remark}
	Note that the above quoted Theorem is slightly different from Theorem 5.5 in \cite{XLYZ}. Indeed, the domain of $g$ in Theorem 5.5 in \cite{XLYZ} is $\sn$ and is changed to $\Omega$ here. Despite the change, the proof, however, works for any $\Omega$ without any essential changes once we realize the fact that if $h:\Omega\rightarrow \mathbb{R}$, then for almost all $x\in \partial [h]$, we have $\nu_{[h]}(x)\in \Omega$. In this exact quoted form, a proof of Theorem \ref{theorem variational formula} can be found in the Appendix of \cite{XYZZ}.
\end{remark}

In the discrete case, Theorem \ref{theorem variational formula} becomes the following.
\begin{coro}
\label{coro discrete variational formula}
	Let $v_1, \dots, v_N$ be $N$ unit vectors that are not contained in any closed hemisphere and $z=(z_1,\dots, z_N)\in (\mathbb{R}_+)^N$. Let $\beta=(\beta_1,\dots, \beta_N)\in \mathbb{R}^N$. For sufficiently small $|t|$, consider $z(t)=z+t\beta$ and
	\begin{equation}
		P_t = [z(t)]=\bigcap_{i=1}^N \{x\in \rn: x\cdot v_i\leq z_i(t)=z_i+t\beta_i\}.
	\end{equation}
	Then, for $q>0$, we have
	\begin{equation}
		\frac{d}{dt}\Big|_{t=0} I_q(P_t)= \sum_{i=1}^N \beta_i F_q(P_0, v_i).
	\end{equation}
\end{coro}
Chord measures inherit their translation invariance and homogeneity (of degree $n+q-2$) from chord integrals. It was shown in \cite{XYZZ} that the chord measure $F_q(K,\cdot)$ is weakly continuous on $\kn$ with respect to Hausdorff metric. Inspired by the much fruitful $L_p$ Brunn-Minkowski theory, it is natural to consider the $L_p$ version of the chord measures. For each $p\in \mathbb{R}$ and $K\in \ko$, the $L_p$ chord measure $F_{p,q}(K,\cdot)$ is defined by
\begin{equation}
	dF_{p,q}(K,v) = h_K(v)^{1-p}dF_q(K,v).
\end{equation}
We remark here that the $L_p$ chord measure naturally arises from replacing the Minkowski sum $K+tL$ in \eqref{eq 10222} by the $L_p$ Minkowski sum $K+_p t\cdot L$. See Xi-LYZ \cite{XLYZ} for details.

For $p\leq 1$, since the exponent $1-p$ is nonnegative, the definition of $F_{p,q}(K,\cdot)$ naturally extends to all $K\in \kn$ as long as $o\in K$. It is important to notice that with the exception of $p=1$, the $L_p$ chord measure loses its translation invariance. However, it is still homogeneous in $K$; that is
\begin{equation*}
	F_{p,q}(tK,\cdot) = t^{n+q-p-1} F_{p,q}(K,\cdot)
\end{equation*}
for each $t>0$. In this paper, we focus our attention on the case $p\in [0,1)$. In this case, it was shown in \cite{XYZZ} that the $L_p$ chord measure is weakly continuous on the set of convex bodies containing the origin (not necessarily as an interior point) with respect to Hausdorff metric.
\begin{prop}\label{prop weak continuity}
	Let $p\in [0,1)$, $q>0$ and $K_i, K\subset \kn$. If $o\in K_i\cap K$ and $K_i$ converges to $K$ in Hausdorff metric, then $F_{p,q}(K_i,\cdot)$ converges weakly to $F_{p,q}(K,\cdot)$.
\end{prop}

In particular, when $q=1$, the $L_p$ chord measures are nothing but the $L_p$ surface area measures fundamental in the $L_p$ Brunn-Minkowski theory. Note that when $p=0$, the $L_0$ surface area enjoys significant geometric meaning---up to a dimensional constant, it is more widely known as \emph{cone volume measure}. Following this, for $q>0$ and $K\in \kn$, we define the cone chord measure of $K$, denoted by $G_q(K,\cdot)$, as the finite Borel measure on $\sn$ given by
\begin{equation}
	G_q(K,\cdot) = \frac{1}{n+q-1} F_{0,q}(K,\cdot).
\end{equation}
It was shown in \cite{XLYZ} that the total measure of $G_q$ is nothing but the chord integral $I_q$; that is
\begin{equation}
	|G_q(K,\cdot)|=I_q(K).
\end{equation}
It is worthwhile to note that when $q=1$, this simply recovers the fact that the total measure of the cone volume measure gives the volume of a convex body.

\section{Polytopes whose normals are in general position}
 \label{sec general position}

The main result in this section demonstrates that if $P$ is a polytope whose outer unit facet normals are in general position, then the size of $P$, if it is large, has to be large uniformly in every direction.

We first gather the following fact about polytopes whose normals are in general position.
\begin{lemma}[Lemma 4.1 in \cite{MR3228445}]
\label{lemma lemma 4.1 in Zhu}
Let $v_1,\dots, v_N$ be $N$ unit vectors that are not contained in any closed hemisphere and $P\in \mathcal{P}(v_1, \dots, v_N)$. Assume that $v_1,\dots, v_N$ are in general position in dimension $n$. Then $F(P, v_i)$ is either a point or a facet. Moreover, if $n\geq 3$ and $F(P, v_i)$ is a facet, then the outer unit normals of $F(P,v_i)$ (viewed as an $(n-1)$-dimensional convex body in the hyperplane containing it) are in general position in dimension $(n-1)$.
\end{lemma}

We need the following trivial lemma.
\begin{lemma}
\label{lemma orthonormal bases compact}
	The set of all orthornormal bases, as a subset of $\sn\times\dots\times \sn$, is compact.
\end{lemma}
\begin{proof}
	Note that $(e_1,\dots, e_n)\in \sn\times\cdots\times\sn$ is an orthonormal basis if and only if $e_i\cdot e_j=0$ for any $i\neq j$. We set
	\begin{equation}
		f(e_1,\dots, e_n) = \sum_{i\neq j}|e_i\cdot e_j|.
	\end{equation}
	It is simple to see $f$ is continuous on $\sn\times\dots\times\sn$ and $(e_1,\dots, e_n)$ is an orthonormal basis if and only if $f(e_1,\dots, e_n)=0$. Hence, the set of all orthonormal bases is a closed subset and being a closed subset of a compact set makes it compact.
\end{proof}

Let $v_1, \dots, v_N$ be $N$ unit vectors that are in general position in dimension $n$. Define
\begin{equation}
	g(e_1,\dots, e_n) =\min_{1\leq i_1<i_2\leq N} \min_{1\leq j\leq n} \max\left\{\sqrt{1-(v_{i_1}\cdot e_j)^2}, \sqrt{1-(v_{i_2}\cdot e_j)^2}\right\}.
\end{equation}
Note that since $v_1, \dots, v_N$ are in general position, for any $1\leq i_1<i_2\leq N$, the vectors $v_{i_1}$ and $v_{i_2}$ are not parallel. Thus, we have
\begin{equation}
	\max\left\{\sqrt{1-(v_{i_1}\cdot e_j)^2}, \sqrt{1-(v_{i_2}\cdot e_j)^2}\right\}>0,
\end{equation}
for any arbitrary unit vector $e_j$.Therefore, we conclude that $g$ is a positive function. It is simple to see that $g$ is also continuous. By Lemma \ref{lemma orthonormal bases compact}, there exists $c_0>0$ such that
\begin{equation}
\label{eq local 1}
	g(e_1,\dots, e_n)\geq c_0,
\end{equation}
if $e_1, \dots, e_n$ forms an orthonormal basis. Note that $c_0$ here only depends on $v_1,\dots, v_N$.

We need the following estimate. Note that to avoid introducing constants that look like $c_{1000}$, we will use $c_0$ to denote a constant that may change from line to line (and certainly from lemma to lemma).
\begin{lemma}
\label{lemma uniform ball estimate}
	Let $v_1,\dots, v_N\in \sn$ be in general position in dimension $n$ and $1\leq i_1<i_2\leq N$. Let $B_1^{n-1}$ and $B_2^{n-1}$ be two $(n-1)$-dimensional balls of radius $R$ such that $B_1^{n-1}\perp v_{i_1}$ and $B_2^{n-1} \perp v_{i_2}$. Consider
	\begin{equation}
		K= \text{conv}\, \{B_1^{n-1}, B_2^{n-1}\}.
	\end{equation}
	Then, there exists $c_0>0$,  and $x_0\in \text{int}\, K$ such that
	\begin{equation}
		B(x_0, c_0R)\subset K.
	\end{equation}
	Here, the constant $c_0>0$ only depends on $n$ and $v_1,\dots, v_N$. In particular, it does not depend on $i_1$ and $i_2$.
\end{lemma}
\begin{proof}
Note that since $v_{i_1}$ and $v_{i_2}$ are linearly independent, the convex set $K$ has to have nonempty interior. By John's lemma, there exists $x_0\in \text{int}\, K$ and $a_1, \dots, a_n>0$ and an orthonormal basis $e_1,\dots, e_n$ such that the ellipsoid
	\begin{equation}
		E = \left\{x\in \rn: \frac{|(x-x_0)\cdot e_1|^2}{a_1^2}+\dots+\frac{|(x-x_0)\cdot e_n|^2}{a_n^2}\leq 1  \right\}
	\end{equation}
	satisfies
	\begin{equation}
	\label{eq local 2}
		E\subset K\subset x_0+n(E-x_0).
	\end{equation}
	For simplicity of notation, we denote $x_0+n(E-x_0)$ by $E'$, which is just an enlargement of $E$ with respect to its center $x_0$ by a factor of $n$.
	
	Since $K\subset E'$, we have
	\begin{equation}
		|P_{e_i} K|\leq |P_{e_i}E'| = 2na_i,
	\end{equation}
	for each $i=1,\dots, n$. Here, we use $P_{e_i} K$ to denote the image of the orthogonal projection of $K$ onto the line spanned by $e_i$. On the other hand, since $B_1^{n-1}\subset K$ and $B_2^{n-1}\subset K$, we have
	\begin{equation}
		|P_{e_i} K|\geq \max \{ |P_{e_i}B_1^{n-1}|, |P_{e_i}B_2^{n-1}|\}.
	\end{equation}
	Note that since $B_1^{n-1}\perp v_{i_1}$ and $B_2^{n-1}\perp v_{i_2}$, we have
	\begin{equation}
		\begin{aligned}
			|P_{e_i} B_1^{n-1}| = 2R\sqrt{1-(v_{i_1}\cdot e_i)^2},\\
			|P_{e_i} B_2^{n-1}| = 2R\sqrt{1-(v_{i_2}\cdot e_i)^2}.
		\end{aligned}
	\end{equation}
	Combining the above, we have
	\begin{equation}
		a_i\geq \frac{R}{n}\max\left\{\sqrt{1-(v_{i_1}\cdot e_i)^2}, \sqrt{1-(v_{i_2}\cdot e_i)^2}\right\}.
	\end{equation}
	By \eqref{eq local 1}, there exists $c_0>0$ independent of the choice of $i_1$ and $i_2$ such that
	\begin{equation}
		a_i\geq c_0R.
	\end{equation}
	By the left half of \eqref{eq local 2}, we have
	\begin{equation}
		B(x_0, c_0R)\subset K.
	\end{equation}

\end{proof}

The following key lemma reveals the special structure for polytopes whose normals are in general position: if the polytope gets large, then it has to get large uniformly in all directions.

\begin{lemma}
\label{lemma uniform unbounded estimate}
	Let $v_1,\dots, v_N$ be $N$ unit vectors that are not contained in any closed hemisphere, and $P_i$ be a sequence of polytopes in $\mathcal{P}(v_1,\dots, v_N)$. Assume the vectors $v_1,\dots, v_N$ are in general position in dimension $n$. If the outer radii $R_i$ of $P_i$ are not uniformly bounded in $i$, then their inner radii $r_i$ are not uniformly bounded in $i$ either.
\end{lemma}
\begin{proof}
	We will do induction on the dimension $n$.
	
	First, let us consider the $n=2$ case. Since $R_i$ are not uniformly bounded in $i$, there exists an edge $E_i$ from each $P_i$ such that $|E_i|$ are not uniformly bounded. Recall that the surface area measure has its centroid at the origin; that is,
	\begin{equation}
		\sum_{j=1}^N |F(P_i, v_j)|v_j = o,
	\end{equation}
	where $F(P_i, v_j) = \{x\in P_i: x\cdot v_j = h_{P_i}(v_j)\}$.
	Therefore, there must exist another edge $E'_i$ (different from $E_i$) of $P_i$ such that $|E_i'|$ are not uniformly bounded either. By taking a subsequence (and without causing confusion, use the same subscript for the subsequence), we can assume
	\begin{equation}
		|E_i|, |E_i'|>2i.
	\end{equation}
	Observe that since $E_i$ and $E_i'$ are edges of $P_i$, there exist $i_1\neq i_2$ such that $v_{i_1}$ and $v_{i_2}$ are the corresponding normals. We now take line segments $L_i$ and $L_i'$ of length $2i$ that are subsets of $E_i$ and $E_i'$, respectively.
	
	Consider
	\begin{equation}
	\label{eq local 3}
		K_i = \text{conv} \left\{L_i, L_{i}'\right\}\subset  \text{conv}\, \{E_i, E_i'\}\subset P_i.
	\end{equation}
	By Lemma \ref{lemma uniform ball estimate}, there exists $c_0>0$ and $x_i\in \text{int}\, K_i\subset \text{int}\, P_i$ such that
	\begin{equation}
		B(x_i, c_0i)\subset K_i.
	\end{equation}
	This, when combined with \eqref{eq local 3}, implies that the inner radii $r_i$ of $P_i$ are not uniformly bounded. This proves the base step.
	
	We now assume that the lemma is true in dimension $(n-1)$ and use that to establish the dimension $n$ case.
	
	Since $R_i$ are not uniformly bounded, there exists a facet $E_i$ from each $P_i$ such that the outer radii $\widetilde{R}_i$ of $E_i$ are not uniformly bounded. By possibly taking a subsequence, we may assume all $E_i$ have the same normal vector; that is, they are parallel. By Lemma \ref{lemma lemma 4.1 in Zhu}, the outer unit normals of $E_i$ are in general position in dimension $(n-1)$. Therefore, by the inductive hypothesis, the inner radii of $E_i$ are not uniformly bounded. In particular, their $(n-1)$-dimensional areas are not uniformly bounded in $i$. Using again the fact that surface area measure has its centroid at the origin, we may find a facet  $E_i'$ from $P_i$ such that the $(n-1)$-dimensional areas of $E_i'$ are not uniformly bounded either; as a consequence, the outer radii $\widetilde{R}'_i$ of $E_i'$ are not uniformly bounded. Repeating the same argument as for $E_i$, we may use the inductive hypothesis to conclude that the inner radii of $E_i'$ are not uniformly bounded.
	
	Observe that since $E_i$ and $E_i'$ are facets of $P_i$, there exist $i_1\neq i_2$ such that $v_{i_1}$ and $v_{i_2}$ are the corresponding normals. Since the inner radii for $E_i$ and $E_i'$ are not uniformly bounded, by taking a subsequence (and using the same notation for the subsequence), we can assume both $E_i$ and $E_i'$ contain $(n-1)$ dimensional balls of radius $i$. We denote these balls by $B_i$ and $B_i'$.
	
	Consider
	\begin{equation}
	\label{eq local 4}
		K_i = \text{conv}\, \{B_i, B_i'\}\subset \text{conv}\, \{E_i, E_i'\}\subset P_i.
	\end{equation}
	By Lemma \ref{lemma uniform ball estimate}, there exists $c_0>0$ and $x_i\in \text{int}\, K_i\subset \text{int}\, P_i$ such that
	\begin{equation}
		B(x_i, c_0i)\subset K_i.
	\end{equation}
	This, when combined with \eqref{eq local 4}, implies that the inner radii $r_i$ of $P_i$ are not uniformly bounded. This completes the proof.
\end{proof}

An immediate consequence of Lemma \ref{lemma uniform unbounded estimate} is the following result in dimensions greater than or equal to $2$.
\begin{coro}
\label{coro C0 estimate}
	Let $v_1, \dots, v_N\in \sn$ be $N$ unit vectors that are not contained in any closed hemisphere and $P\in \mathcal{P}(v_1,\dots, v_N)$. Assume that $v_1, \dots, v_N$ are in general position in dimension $n$. If the outer radius $R_i$ of $P_i$ is not uniformly bounded and $q\geq 0$, then the $q$-th chord integral $I_q(P_i)$ is also unbounded.
\end{coro}
\begin{proof}
	This follows immediately from Lemma \ref{lemma uniform unbounded estimate}, the homogeneity and the translation invariance of $I_q$, and the fact that $I_q(B)$ is positive for the centered unit ball $B$.
\end{proof}

\section{The discrete $L_p$ chord Minkowski problem}
\label{sec discrete}

Let $\mu$ be a finite discrete Borel measure on $\sn$ that is not concentrated in any closed hemisphere; that is
\begin{equation}
\label{eq 1042}
	\mu = \sum_{i=1}^N \alpha_i \delta_{v_{i}},
\end{equation}
for some $\alpha_i>0$ and unit vectors $v_1, \dots v_N\in \sn$ not contained in any closed hemisphere.

In this section, we will solve the discrete $L_p$ chord Minkowski problem for $q>0$ and $0\leq p<1$.

For any $z=(z_1,\dots, z_N)\in \mathbb{R}^N$ such that $[z]$ has nonempty interior, we define
\begin{equation*}
	\Phi_{p,\mu}(z,\xi) =\begin{cases}
		\sum_{j=1}^N (z_j-\xi\cdot v_j)^p\cdot \alpha_j, &\text{ if } p\in (0,1),\\
		\sum_{j=1}^N \log (z_j-\xi\cdot v_j)\cdot \alpha_j, &\text{ if } p=0,\\
	\end{cases}
\end{equation*}
for each $\xi \in [z]$. We adopt the convention that $\log 0=-\infty$. When there is no confusion about what the underlying measure $\mu$ is, we shall write $\Phi_p=\Phi_{p,\mu}$.

It is simple to see that for each $p\in [0,1)$, the functional $\Phi_{p}(z, \cdot)$ is strictly concave in $\xi\in \Int [z]$. Therefore, the maximizer to the problem
\begin{equation}
	\sup_{\xi \in [z]}\Phi_{p}(z,\xi),
\end{equation}
if it exists, must be unique. When $p=0$, the existence of the maximizer $\xi \in \Int [z]$ follows from the fact that if a sequence of interior points $\Int [z]\ni\xi_j\rightarrow \partial [z]$, then $\Phi_0(z,\xi_j)\rightarrow -\infty$. This follows from the trivial fact that $\log 0 = -\infty$. In the case $p\in (0,1)$, the existence of maximizer is less trivial and was shown in Zhu \cite{MR3352764}. We summarize both the $p=0$ and $p\in (0,1)$ case in the following lemma.

\begin{lemma}[\cite{MR3352764}]
	\label{lemma phi has unique maximizer}
	Let $z=(z_1,\dots, z_N)\in \mathbb{R}^N$ be such that $[z]$ has nonempty interior and $p\in [0,1)$. Then the maximizer of the following optimization problem
	\begin{equation}
		\sup_{\xi \in [z]}\Phi_p(z,\xi)
	\end{equation}
	is uniquely attained at some $\xi_0\in \text{int}\,[z]$.
\end{lemma}

We will use $\xi_{p,\mu}(z)$ to denote the unique maximizer in Lemma \ref{lemma phi has unique maximizer}. Similar to before, when the context is clear, we will suppress the subscript $\mu$.

It is simple to observe that the operator $\xi_p$ is homogeneous of degree $1$. That is, for any $\lambda>0$, we have $\xi_p(\lambda z)= \lambda \xi_p(z)$.

The following fact regarding the continuity of $\xi_p$ in $z$ is well-known. For the sake of completeness, we provide a quick proof.
\begin{lemma}
\label{lemma continuity of xi}
	Let $z_l\in \mathbb{R}^N$ be such that $\lim_{l\rightarrow \infty}z_l=z\in \mathbb{R}^N$ and $p\in [0,1)$. If $[z]$ has nonempty interior, then
	\begin{equation}
	\label{eq local 7}
		\lim_{l\rightarrow \infty} \xi_p(z_l) =\xi_p(z),
	\end{equation}
	and
	\begin{equation}
	\label{eq local 6}
		\lim_{l\rightarrow \infty} \Phi_p(z_l, \xi_p(z_l)) = \Phi_p(z, \xi_p(z)).
	\end{equation}
\end{lemma}
\begin{proof}
We first note that \eqref{eq local 6} is a direct consequence of \eqref{eq local 7} by the definition of $\Phi_p$. Therefore, only \eqref{eq local 7} requires a proof.

By the fact that $z_l\rightarrow z$, the assumption that $[z]$ has nonempty interior, and Lemma \ref{lemma phi has unique maximizer}, both $\xi_p(z_l)$ for sufficiently large $l$  and $\xi_p(z)$ are well-defined. Moreover, we can conclude from the fact that $z_l\rightarrow z$ and the fact that $\xi_p(z_l)\in \Int [z_l]$ that $\xi_p(z_l)$ are uniformly bounded in $l$. Therefore, if \eqref{eq local 7} is false, there must exist a subsequence (which we still denote as $\xi_p(z_l)$) such that
\begin{equation*}
	\xi_p(z_l)\rightarrow \xi'\neq \xi_p(z).
\end{equation*}
Note that it must be the case that $\xi'\in [z]$. Moreover, by the definition of $\Phi_p$ and Lemma \ref{lemma phi has unique maximizer},
\begin{equation}
	\lim_{l\rightarrow \infty} \Phi_p(z_l, \xi_p(z_l))=\Phi_p(z,\xi')<\Phi_p(z,\xi_p(z))= \lim_{l\rightarrow \infty}\Phi_p(z_l,\xi_p(z)).
\end{equation}
However,
\begin{equation}
	\lim_{l\rightarrow \infty}\Phi_p(z_l,\xi_p(z))\leq \lim_{l\rightarrow \infty}\Phi_p(z_l,\xi_p(z_l))=\Phi_p(z,\xi').
\end{equation}
The above contradiction immediately gives the desired result.
\end{proof}

The next lemma shows that $\xi_p(z)$ is a differentiable function with respect to vector addition in $z$.

\begin{lemma}
\label{lemma differentiable}
	Let $z=(z_1,\dots, z_N)\in \mathbb{R}^N_+$, $p\in [0,1)$, and $\mu$ be as given in \eqref{eq 1042}. For each $\beta\in \mathbb{R}^N$, consider
	\begin{equation*}
		z(t) = z+t\beta,
	\end{equation*}
	for sufficiently small $|t|$ so that $z(t)\in \mathbb{R}^N_+$. Denote $\xi_p(t)= \xi_p(z(t))$. If $\xi_p(0)=o$, then $\xi_p'(0)$ exists. Moreover,
	\begin{equation}
	\label{eq 1043}
		\begin{aligned}
			o=\sum_{j=1}^N z_j^{p-1}\alpha_j v_j.
		\end{aligned}
	\end{equation}
\end{lemma}
\begin{proof}
	Since $\xi_p(t)\in \Int [z(t)]$ and maximizes
	\begin{equation}
		\sup_{\xi \in [z(t)]} \Phi_{p}(z(t), \xi),
	\end{equation}
	taking the derivative in $\xi$ shows
	\begin{equation}
	\label{eq 1041}
	o =\sum_{j=1}^N (z_j(t)-\xi_p(t)\cdot v_j)^{p-1}\alpha_j v_j.
	\end{equation}
	In particular, at $t=0$, we have
	\begin{equation}
	o=\sum_{j=1}^N z_j^{p-1}\alpha_j v_j.
	\end{equation}
which establishes \eqref{eq 1043}.

Set
\begin{equation}
	F_p(t,\xi)=\sum_{j=1}^N (z_j(t)-\xi\cdot v_j)^{p-1}\alpha_j v_j.
	\end{equation}

Then, \eqref{eq 1041} simply says
\begin{equation}
	F_p(t,\xi_p(t))=o.
\end{equation}
By a direct computation, the Jacobian with respect to $\xi$ of $F_p$ at $t=0$ and $\xi=o$ is
\begin{equation}
\left.\frac{\partial F_p}{\partial \xi}\right|_{(0,o)} =(1-p)\sum_{j=1}^N z_j^{p-2} \alpha_j v_j\otimes v_j.
\end{equation}
Since $v_1,\dots, v_N$ span $\rn$, we conclude that the Jacobian $\frac{\partial F_p}{\partial \xi}$ is positive-definite at $t=0$ and $\xi=o$. By the implicit function theorem, we conclude that $\xi_p'(0)$ exists.
\end{proof}

For each $0\leq p<1$ and $q>0$, we consider the optimization problem
\begin{equation}
\label{eq 1051}
	\inf\{\Phi_p(z, \xi_p(z)): z\in \mathbb{R}^N, I_q([z])=|\mu|\}.
\end{equation}

\begin{lemma}
\label{lemma 1041}
	Let $0\leq p<1$ and $q>0$. If there exists $z\in \mathbb{R}^N_+$ with $\xi_p(z)=o$ and $I_q([z])=|\mu|$ satisfying
	\begin{equation}
		\Phi_p(z,o)=\inf\{\Phi_p(z, \xi_p(z)): z\in \mathbb{R}^N, I_q([z])=|\mu|\},
	\end{equation}
	then, there exists a polytope $P\in \mathcal{P}(v_1,\dots, v_N)$ containing the origin in its interior such that
	\begin{equation}
	\label{eq 1046}
		F_{p,q}(P, \cdot)=\mu.
	\end{equation}
	Moreover, for each $i=1,\dots, N$, we have
	\begin{equation}
	\label{eq 791}
		h_{[z]}(v_i)=z_i.
	\end{equation}
\end{lemma}
\begin{proof}
	Because of homogeneity, we may assume $|\mu|=1$. Let $\beta\in \mathbb{R}^N$ be arbitrary and set
	\begin{equation}
		z(t) = z+t\beta.
	\end{equation}
	For sufficiently small $|t|$, we have $z(t)\in \mathbb{R}^N_+$. Set
	\begin{equation}
		\lambda(t) = I_q([z(t)])^{-\frac{1}{n+q-1}}.
	\end{equation}
	Note that $\lambda(0)=1$. By homogeneity of $I_q$, it is apparent that $I_q([\lambda(t)z(t)])=1$. By Corollary \ref{coro discrete variational formula}, we have
	\begin{equation}
	\label{eq 1044}
		\lambda'(0)=-\frac{1}{n+q-1} \sum_{i=1}^N \beta_i F_q([z], v_i).
	\end{equation}
	Let $\xi_p(t) = \xi_p(\lambda(t)z(t))=\lambda(t)\xi_p(z(t))$ and
	\begin{equation}
		\Psi_p(t) = \Phi_{p}(\lambda(t)z(t), \xi_p(t)).
	\end{equation}
	By Lemma \ref{lemma differentiable}, $\xi_p$ is differentiable at $t=0$. Moreover, \eqref{eq 1043} holds.

	Since $z$ is a minimizer, the fact that $0=\Psi_p'(0)$ shows
	\begin{equation}
	0=\lambda'(0) \left(\sum_{j=1}^N z_j^p \alpha_j\right) + \sum_{i=1}^N	z_i^{p-1} \alpha_i\beta_i-\xi_p'(0) \cdot \left(\sum_{j=1}^N z_j^{p-1}\alpha_jv_j\right).
	\end{equation}
	By \eqref{eq 1043} and \eqref{eq 1044}, we have
	\begin{equation}
		0=-\frac{1}{n+q-1}\left(\sum_{j=1}^Nz_j^p\alpha_j\right)\sum_{i=1}^N\beta_i F_q([z], v_i)+\sum_{i=1}^N \beta_i z_i^{p-1}\alpha_i.
	\end{equation}
	Since $\beta$ is arbitrary, we conclude that
	\begin{equation}
		F_{p,q}([z], \cdot) = \frac{n+q-1}{\Phi_p(z,o)}\mu(\cdot),
	\end{equation}
	if $p\in (0,1)$ and
	\begin{equation}
		G_q([z])=\mu(\cdot).
	\end{equation}
	The existence of $P$ now immediately follows from the fact that $F_{p,q}(K,\cdot)$ is homogeneous of degree $n+q-1-p\neq 0$ in $K$.
	
	We now show \eqref{eq 791}. Assume that it fails for some $i_0$. Let $\widetilde{z}\in \mathbb{R}^N_+$ be such that $\widetilde{z}_i=h_{[z]}(v_i)$. 
	By \eqref{eq 6151}, we have $\widetilde{z}_{i_0}<z_{i_0}$ and $\widetilde{z}_{i}\leq z_{i}$ for $i\neq i_0$. Note that $[z]=[\widetilde{z}]$ and consequently, $I_q([\widetilde{z}])=|\mu|$.
	By definition of $\Phi_p$ and $\xi_p$, we have
	\begin{equation*}
		\Phi_p(\widetilde{z}, \xi_p(\widetilde{z}))<\Phi_p(z, \xi_p(\widetilde{z}))\leq \Phi_p(z,\xi_p(z))=\Phi_p(z,o).
	\end{equation*}
	This is a contradiction to $z$ being a minimizer.
\end{proof}

We are in position to solve the discrete $L_p$ chord Minkowski problem when $0\leq p<1$ and $q>0$.

\begin{theorem}
\label{thm discrete}
	Let $0\leq p<1$, $q>0$, and $\mu$ be as given in \eqref{eq 1042}. If $v_1,\dots v_N\in \sn$ are in general position in dimension $n$, then there exists a polytope $P\in \mathcal{P}(v_1,\dots, v_N)$ containing the origin in its interior such that
	\begin{equation}
		F_{p,q}(P,\cdot)=\mu.
	\end{equation}
\end{theorem}
\begin{proof}
	We consider the minimization problem \eqref{eq 1051}. Let $z^l\in \mathbb{R}^N$ be a minimizing sequence; that is $I_q([z^l])=|\mu|$ and
	\begin{equation}
		\lim_{l\rightarrow \infty} \Phi_p(z^l, \xi_p(z^l)) = \inf\{\Phi_p(z,\xi_p(z)): z\in \mathbb{R}^N, I_q([z])=|\mu|\}.
	\end{equation}
	Note that by translation invariance of $I_q$ and the simple fact that $\Phi_p(z, \xi)=\Phi_p(z', o)$ where $z'_j=z_j-\xi\cdot v_j$, we can assume without loss of generality that $\xi_p(z^l)=o$. Moreover, by the definition of $\Phi_p$, it must be the case that
	\begin{equation}
	\label{eq 1052}
		z^l_j = h_{[z^l]}(v_j).
	\end{equation}
	The fact that $o=\xi_p(z^l)\in \Int [z^l]$ now implies that $z_j^l>0$. Since $I_q([z^l])=|\mu|$ is finite, Corollary \ref{coro C0 estimate} implies that the outer radii of $[z^l]$ are uniformly bounded. This, when combined with the fact that $o\in [z^l]$, implies that $[z^l]$ is uniformly bounded, which by \eqref{eq 1052} implies that $z^l$ is uniformly bounded in $\mathbb{R}^N$ in $l$. Therefore, we may (by potentially taking a subsequence) assume that $z^l\rightarrow z^0$ for some $z^0\in \mathbb{R}^N$. By continuity of $I_q$, we have $I_q([z^0])=|\mu|$, which implies that $[z^0]$ contains nonempty interior. Lemma \ref{lemma continuity of xi} now implies that $\xi_p(z^0) = \lim_{l\rightarrow \infty}\xi_p(z^l)=o$. This and the fact that $\xi_p(z^0)\in \Int [z^0]$ imply that $z^0\in \mathbb{R}_+^N$. Moreover, by the definition of $\Phi_p$, we have
	\begin{equation}
		\Phi_p(z^0,o)=\lim_{l\rightarrow \infty} \Phi_p(z^l, o) = \inf\{\Phi_p(z,\xi_p(z)): z\in \mathbb{R}^N, I_q([z])=|\mu|\}.
	\end{equation}
	Lemma \ref{lemma 1041} now implies the existence of $P$.
\end{proof}

When $0<p<1$, Theorem \ref{thm discrete} in fact holds even without the assumption that $v_1,\dots v_N\in \sn$ are in general position in dimension $n$.
\begin{theorem}
\label{thm discrete'}
	Let $0< p<1$, $q>0$, and $\mu$ be as given in \eqref{eq 1042}. Then there exists a polytope $P\in \mathcal{P}(v_1,\dots, v_N)$ containing the origin in its interior such that
	\begin{equation}
		F_{p,q}(P,\cdot)=\mu.
	\end{equation}
\end{theorem}
\begin{proof}
	The proof for Theorem \ref{thm discrete} remains valid aside from the fact that we can no longer use Corollary \ref{coro C0 estimate} to show $z^l$ is uniformly bounded in $\mathbb{R}^N$.
	
	We show in this proof that in the case of $0<p<1$, the uniform boundedness of $z^l$ can still be obtained.
	
	Set $\zeta(r)=(r,r,\dots, r)\in \mathbb{R}^N$. Then, by the homogeneity of $I_q$, we may find $r_0>0$ such that $I_q([\zeta(r_0)])=|\mu|$. Therefore,
	\begin{equation}
	\label{eq 10610}
		\lim_{l\rightarrow \infty} \Phi_p(z^l, o)\leq \Phi_p(\zeta(r_0), \xi_p(\zeta(r_0)))= \sum_{j=1}^N \left(r_0-\xi_p\left(\zeta(r_0)\right)\cdot v_j\right)^p\alpha_j\leq \sum_{j=1}^N (2r_0)^p\alpha_j <\infty,
	\end{equation}
	where we used the fact that $\xi_p(\zeta(r_0))\in \Int [\zeta(r_0)]$ (by Lemma \ref{lemma phi has unique maximizer}).
	
	On the other hand, if we set $L_l=\max_j z_j^l$, then
	\begin{equation}
	\label{eq 10611}
		\Phi_p(z^l, o)= \sum_{j=1}^n (z^l_j)^p \alpha_j \geq L_l^p \min_j \alpha_j.
	\end{equation}
	The uniform boundedness of $z^l$ now comes from \eqref{eq 10610}, \eqref{eq 10611}, and the definition of $L_l$.
\end{proof}
\begin{remark}
	The proof for uniform upper bound in Theorem \ref{thm discrete'} would not work for $p=0$ since the logarithm function takes both positive and negative values.
\end{remark}

\section{The general case}

Throughout this section, we assume $p\in [0,1)$ and $q>0$ unless otherwise specified. Recall that, for $\Omega=\{v_1,\dots, v_N\}$ and for every $z\in \mathbb{R}^{N}$, we use $[z, \Omega]$ to denote the Wulff shape generated by $z$ on $\Omega$; that is
\begin{equation}
	[z,\Omega] = \{x\in \rn: x\cdot v_i\leq z_i, \quad i=1,\dots, N\}.
\end{equation}
For the purpose of the section, we need to explicitly mention the different underlying $\Omega$ in different expressions (as they change between contexts).

 Let $\mu$ be a finite Borel measure (not necessarily discrete) on $\sn$ that is not concentrated in any closed hemisphere. The purpose of the section is to solve the $L_p$ chord Minkowski problem for $\mu$; that is, to solve
\begin{equation}
	F_{p,q}(K,\cdot) = \mu,
\end{equation}
via an approximation scheme based on the polytopal solution we obtained in Section \ref{sec discrete}.

We first construct a sequence of discrete measures whose support sets are in general position such that the sequence of discrete measures converges to $\mu$ weakly.

For each positive integer $m$, it is simple to see that there is a way to partition $\sn$ into sufficiently many pieces so that the diameter of each small piece is less than $\frac{1}{m}$; that is, there exists $\mathcal{N}_m>0$ and a partition of $\sn$, denoted by $U_{1,m},\dots, U_{\mathcal{N}_m, m}$ such that $d(U_{i,m})<\frac{1}{m}$ and $U_{i,m}$ contains nonempty interior (relative to the topology of $\sn$). We may choose $v_{i,m}\in U_{i,m}$ so that $v_{1,m},\dots, v_{\mathcal{N}_m, m}$ are in general position. When $m$ is large, it is clear that the vectors $v_{1,m},\dots, v_{\mathcal{N}_m, m}$ cannot be contained in any closed hemisphere.

We define the discrete measure $\mu_m$ on $\sn$ by
\begin{equation}
	\mu_m=\sum_{i=1}^{\mathcal{N}_m} \left(\mu(U_{i,m})+\frac{1}{\mathcal{N}_m^2}\right)\delta_{v_{i,m}},
\end{equation}
and
\begin{equation}
\label{eq 1059}
	\overline{\mu}_m = \frac{|\mu|}{|\mu_m|} \mu_m.
\end{equation}
Denote by $\Omega_m$ the support of the discrete measure $\mu_m$; that is,
\begin{equation}
	\Omega_m=\{v_{1,m},\dots, v_{\mathcal{N}_m,m}\}\subset \sn.
\end{equation}
It is clear that $\overline{\mu}_m$ is a discrete measure on $\sn$ satisfying the conditions in Theorem \ref{thm discrete} and $\overline{\mu}_m\rightharpoonup \mu$ weakly. Therefore, by Theorem \ref{thm discrete}, there exist polytopes $P_m$ containing the origin in their interiors such that
\begin{equation}
\label{eq 1054}
	F_{p,q}(P_m, \cdot) = \overline{\mu}_m.
\end{equation}

A careful examination of the proofs for Theorem \ref{thm discrete} and Lemma \ref{lemma 1041} immediately reveals that $P_m$ is a rescaled version of $[z^m,\Omega_m]$ where $z^m\in \mathbb{R}_+^{\mathcal{N}_m}$ satisfies $\xi_{p, \overline{\mu}_m}(z^m)=o$, $I_q([z^m,\Omega_m])=|\mu_m|$ and
\begin{equation}
\label{eq 6145}
		\Phi_{p,\overline{\mu}_m}(z^m,o)=\inf\{\Phi_{p,\overline{\mu}_m}(z,\xi_{p,\overline{\mu}_m}(z)): z\in \mathbb{R}^{\mathcal{N}_m}, I_q([z,\Omega_m])=|\mu_m|\}.
	\end{equation}
In particular, when $|\mu|=1$ (and consequently $|\overline{\mu}_m|=1$), we have,
	\begin{equation}
	\label{eq 10410}
		P_m=\begin{cases}
			\displaystyle\left(\frac{\Phi_{p,\overline{\mu}_m}(z^m,o)}{n+q-1}\right)^{\frac{1}{n+q-p-1}}[z^m,\Omega_m], &\text{ if } p\in (0,1),\\
			\displaystyle\left(\frac{1}{n+q-1}\right)^{\frac{1}{n+q-1}}[z^m,\Omega_m] &\text{ if } p=0.
		\end{cases}
	\end{equation}

\begin{lemma}
\label{lemma 1051}
	If $P_m$ in \eqref{eq 1054} are uniformly bounded and $I_q(P_m)>c_0$ for some constant $c_0>0$, then there exists a convex body $K \in \mathcal{K}^n$ with $o\in K$ such that
	\begin{equation}
	\label{eq 1056}
		F_{p,q}(K,\cdot) = \mu.
	\end{equation}
\end{lemma}
\begin{proof}
	By the Blaschke selection theorem, there exists a subsequence $P_{m_j}$ such that $P_{m_j}\rightarrow K$ for some compact convex set $K$ containing the origin. By the continuity of $I_q$ and the fact that $I_q(P_m)>c_0$, we have $I_q(K)>0$. This in turn implies that $K$ has nonempty interior. Equation \eqref{eq 1056} now readily follows from taking the limit of \eqref{eq 1054} on both sides and Proposition \ref{prop weak continuity}.
\end{proof}

We require the following lemma.

\begin{lemma}
\label{lemma ball bound}
	Let $v_{1,m}, \dots v_{\mathcal{N}_m, m}\in \sn$ be as given above. Consider
	\begin{equation}
	\label{eq 10510}
		Q_m = \bigcap_{i=1}^{\mathcal{N}_m}\{x\in \rn: x\cdot v_{i,m}\leq 1\}.
	\end{equation}
	Then, for sufficiently large $m$, we have
	\begin{equation}
	\label{eq 1057}
		B\subset Q_m\subset 2B,
	\end{equation}
	where $B$ is the centered unit ball.
\end{lemma}
\begin{proof}
	Only the right side of \eqref{eq 1057} requires a proof.
	
	For each $u\in \sn$, since $U_{i,m}$ forms a partition of $\sn$, there must exist $i_m$ such that $u\in U_{i_m,m}$. Recall that $d(U_{i,m})<\frac{1}{m}$. Hence, we may choose $N_0>0$ (independent of $u$) such that for each $m>N_0$,
	\begin{equation}
		u\cdot v_{i_m,m}>1/2.
	\end{equation}
	Since $\rho_{Q_m}(u)u\in Q_m$, we have
	\begin{equation}
		\rho_{Q_m}(u)/2<\rho_{Q_m}(u)u \cdot v_{i_m,m}\leq 1.
	\end{equation}
	Hence $\rho_{Q_m}<2$ for each $m>N_0$, which proves the desired inequality.
\end{proof}

With a slight abuse of notation, for $\xi\in K$, we will write
\begin{equation}
	\Phi_{p, \mu}(K,\xi) = \begin{cases}
		\int_{\sn} h_{K-\xi}^p d\mu, &\text{ if } p\in (0,1),\\
		\int_{\sn} \log h_{K-\xi}d\mu, &\text{ if } p=0.
	\end{cases}
\end{equation}
Note that when $\mu$ is a discrete measure, $\Omega=\{v_1,\dots, v_N\}$ is the support of $\mu$,  and $z\in \mathbb{R}^N$ satisfies $z_j = h_{[z,\Omega]}(v_i)$, we have 
\begin{equation}
	\Phi_{p,\mu}([z,\Omega], \xi) =
	\Phi_{p,\mu}(z, \xi).
\end{equation}
That is: in this special case, $\Phi_{p,\mu}([z,\Omega], \xi)$ is precisely $\Phi_{p,\mu}(z,\xi)$ defined in Section \ref{sec discrete}. 

With the help of Lemma \ref{lemma ball bound}, we have the following estimate.
\begin{lemma}
\label{lemma bound phi}
	Let $P_m$ be as given in \eqref{eq 1054} and $z^m$ be the minimizer to \eqref{eq 6145} with $\xi_{p,\overline{\mu}_m}(z^m)=0$.
	If $|\mu|=1$ (and consequently $|\overline{\mu}_m|=1$), then there exists $c_0>0$ independent of $m$, such that
	\begin{equation}
		\Phi_{p, \overline{\mu}_m}(P_m,o)<c_0.
	\end{equation}
\end{lemma}
\begin{proof}
	Let $Q_m$ be as given in \eqref{eq 10510}. Consider $rQ_m$ for $r>0$. Note that by Lemma \ref{lemma ball bound}, for sufficiently large $m$, we have $rB\subset rQ_m\subset 2rB$. By the homogeneity of $I_q$, there exists $r_0(m)>0$ such that
\begin{equation}
	I_q(r_0(m)Q_m)=1.
\end{equation}
Since $rB\subset rQ_m$, we have
\begin{equation}
	r_0(m)^{n+q-1}I_q(B)=I_q(r_0(m)B)\leq I_q(r_0(m)Q_m)=1.
\end{equation}
Therefore, $r_0(m)\leq r_0$ for some constant $r_0$ independent of $m$.

Since $z^m$ is a minimizer and using the fact that $r_0(m)Q_m\subset 2r_0(m)B\subset 2r_0B$, we have
\begin{equation}
\label{eq 1047}
\begin{aligned}
	&\Phi_{p,\overline{\mu}_m}(z^m,o)\\
	&\leq \begin{cases}
		\displaystyle\int_\sn h_{r_0(m)Q_m-\xi_{p,\overline{\mu}_m}(r_0(m)Q_m)}^pd\overline{\mu}_m \leq \int_\sn h^p_{4r_0B} d\overline{\mu}_m = (4r_0)^p, &\text{ if } p\in (0,1),\\
		\displaystyle\int_{\sn} \log h_{r_0(m)Q_m-\xi_{0,\overline{\mu}_m}(r_0(m)Q_m)}d\overline{\mu}_m\leq \int_\sn \log h_{4r_0B} d\overline{\mu}_m = \log ((4r_0)),&\text{ if } p=0.
	\end{cases}
\end{aligned}	
\end{equation}
The desired bound now follows from \eqref{eq 10410} and the definition of $\Phi_{p,\overline{\mu}_m}$.
\end{proof}

\begin{remark}
\label{remark 1061}
	In the proof of Lemma \ref{lemma bound phi}, in fact, we have shown something stronger:
	\begin{equation}
		\Phi_{p,\overline{\mu}_m}(z^m,o)
	\end{equation}
	is also uniformly bounded from above. Here $z^m\in \mathbb{R}_+^{\mathcal{N}_m}$ with $\xi_{p, \overline{\mu}_m}(z^m)=o$ is the minimizer to \eqref{eq 6145}.
\end{remark}

The rest of the section is devoted to verifying the hypotheses in Lemma \ref{lemma 1051}. Since there is a major difference between the $p=0$ case and the $0<p<1$ case, we shall prove them separately in two different subsections.

\subsection{The $0<p<1$ case}
Throughout this subsection, we assume $0<p<1$ and $q>0$, both of which are fixed.

It is a well-known fact that for each finite Borel measure $\mu$ on $\sn$ that is not concentrated in any closed hemisphere, there exists a constant $\mathfrak{C}_p(\mu)>0$ such that
\begin{equation}
	\int_{\sn}(u\cdot v)_+^pd\mu(v)\geq \mathfrak{C}_p(\mu),
\end{equation}
uniformly for each $u\in \sn$. We prove in the next lemma that for our choice of $\overline{\mu}_m$, the constants $\mathfrak{C}_p(\overline{\mu}_m)$ can be chosen uniformly.

\begin{lemma}
\label{lemma uniform control of Lp}
	Let $\overline{\mu}_m$ be as given in \eqref{eq 1059}. Then, for sufficiently large $m$, we have
	\begin{equation}
		\int_{\sn}(u\cdot v)_+^pd\overline{\mu}_m(v)\geq \frac{1}{2} \mathfrak{C}_p(\mu).
	\end{equation}
\end{lemma}
\begin{proof}
	For notational simplicity, let
	\begin{equation}
		g_m(u) =  \int_{\sn}(u\cdot v)_+^pd\mu_m(v),
	\end{equation}
	\begin{equation}
		\overline{g}_m(u)=\int_{\sn}(u\cdot v)_+^pd\overline{\mu}_m(v) = \frac{|\mu|}{|\mu_m|}g_m(u),
	\end{equation}
	and
	\begin{equation}
		g(u) = \int_{\sn}(u\cdot v)_+^pd\mu(v).
	\end{equation}
	Note that $g\geq \mathfrak{C}_p(\mu)$ and as a consequence, it suffices to show $\overline{g}_m\rightrightarrows g$. To do that, one only needs to show $g_m\rightrightarrows g$.
	
	Let $\varepsilon>0$ be arbitrary.
	
	Note that the function $(u,v)\mapsto (u\cdot v)_+^{p}$ is uniformly continuous on $\sn\times \sn$. Therefore, for sufficiently large $m$ (independent of the choice of $u$), we have
	\begin{equation}
		|(u\cdot v)_+^{p}-(u\cdot v_{i,m})_+^{p}|<\varepsilon,
	\end{equation}
	for each $v\in U_{i,m}$. As a consequence,
	\begin{equation}
		\begin{aligned}
			|g_m(u)-g(u)|&=\left|\sum_{i=1}^{\mathcal{N}_m}\left[(u\cdot v_{i,m})_+^{p} \left(\mu(U_{i,m})+\frac{1}{\mathcal{N}_m^2}\right)-\int_{U_{i,m}}(u\cdot v)_+^{p}d\mu\right]\right|\\
			&\leq \varepsilon |\mu|+\frac{1}{\mathcal{N}_m}.
		\end{aligned}
	\end{equation}
	Note that the above estimate is independent of $u$. Since $\mathcal{N}_m\rightarrow \infty$, we conclude the desired uniform convergence.
\end{proof}

\begin{lemma}
\label{lemma 1052}
	If  $|\mu|=1$ (and consequently $|\overline{\mu}_m|=1$), the polytopes $P_m$ obtained in \eqref{eq 1054} are uniformly bounded and there exists $c_0>0$ such that $I_q(P_m)>c_0$.
\end{lemma}

\begin{proof}
We fix an arbitrary $m$ that is sufficiently large and prove that the desired bounds for $P_m$ can be chosen independent of $m$.

We first prove that $P_m$ are uniformly bounded (from above).

Let $ L(m)= \max_{\sn} h_{P_m}$. Then by definition of $\Phi_{p,\overline{\mu}_m}$, we have, for some $u\in \sn$,
\begin{equation}
\label{eq 1048}
	\Phi_{p,\overline{\mu}_m}(P_m,o) = \int_{\sn} h_{P_m}(v)^p d\overline{\mu}_m(v)\geq \int_{\sn} (L(m)u\cdot v)_+^p d\overline{\mu}_m(v)\geq L(m)^p\frac{1}{2}\mathfrak{C}_p(\mu),
\end{equation}
owing to Lemma \ref{lemma uniform control of Lp}. By Lemma \ref{lemma bound phi} and \eqref{eq 1048}, we conclude that $L(m)$ is uniformly bounded from above in $m$. By definition of $L(m)$, this in turn implies the uniform boundedness of $P_m$ (from above).

Let $z^m$ be the minimizer to \eqref{eq 6145} with $\xi_{p,\overline{\mu}_m}(z^m)=0$ and $L'(m)=\max_{\sn} h_{[z^m,\Omega_m]}$. Repeating the same argument, we have
\begin{equation}
\label{eq 1064}
	\Phi_{p,\overline{\mu}_m}(z^m,o)\geq \Phi_{p,\overline{\mu}_m}([z^m,\Omega_m],o)\geq L'(m)^p\frac{1}{2}\mathfrak{C}_p(\mu).
\end{equation}
Note here that in the first inequality, we used \eqref{eq 6151}. By Remark \ref{remark 1061}, we conclude that $[z^m,\Omega_m]$ also has a  uniform upper bound. Note that $I_q([z^m,\Omega_m])=1$. This implies that there must exist $c_*>0$ such that $[z^m,\Omega_m]$ contains a ball of radius $c_*$. In turn, since $o\in \Int [z^m,\Omega_m]$, this implies $L'(m)\geq c_*$ and as a consequence of \eqref{eq 1064}, we obtain a uniform lower bound for $\Phi_{p,\overline{\mu}_m}(z^m,o)$. The existence of $c_0$ now follows from \eqref{eq 10410} and the translation-invariance and monotonicity of $I_q$.
\end{proof}

\begin{theorem}\label{thm main thm p neq 0}
	Let $0<p<1$, $q>0$, and $\mu$ be a finite Borel measure on $\sn$ not concentrated in any closed hemisphere. Then, there exists $K\in \mathcal{K}^n$ with $o\in K$ such that
	\begin{equation}
		F_{p,q}(K,\cdot) = \mu.
	\end{equation}
\end{theorem}
\begin{proof}
	By homogeneity of $F_{p,q}$, it suffices to prove the case when $|\mu|=1$. In this case, the result follows immediately from Lemmas \ref{lemma 1051} and \ref{lemma 1052}.
\end{proof}

\subsection{The case $p=0$} \label{sec general p=0}The desired uniform bounds on $P_m$ in the case $p=0$ are much more complicated. This is caused by the fact that a uniform estimate as in Lemma \ref{lemma uniform control of Lp} is unavailable for the integral
\begin{equation}
	\int_{\sn}\log (u\cdot v)_+d\overline{\mu}_m(v).
\end{equation}
In fact, the above integral could well go to $-\infty$.

It turns out that the chord log-Minkowski problem (or the chord $L_0$-Minkowski problem) is heavily connected to subspace mass concentration phenomenon.

Throughout the rest of the section, we assume $1<q<n+1$ is fixed. We say that a given finite Borel measure $\mu$ satisfies the \emph{subspace mass inequality} if
\begin{equation}
\label{eq 1061}
	\frac{\mu(\xi_i\cap \sn)}{|\mu|} <\frac{i+\min\{i, q-1\}}{n+q-1},
\end{equation}
for each $i$ dimensional subspace $\xi_i\subset \rn$ and each $i=1, \dots, n-1$.

It was shown in Xi-LYZ \cite{XLYZ} that when restricting to origin-symmetric cases, the above subspace mass inequality is sufficient for the existence of solutions to the chord log-Minkowski problem:
\begin{theorem}[\cite{XLYZ}]
	Let $1<q<n+1$. If $\mu$ is an even finite Borel measure on $\sn$ that satisfies \eqref{eq 1061}, then there exists an origin-symmetric convex body $K$ in $\rn$ such that
	\begin{equation}
		G_q(K,\cdot)=\mu.
	\end{equation}
\end{theorem}

In this section, we show that the above theorem remains true without symmetric assumptions by employing an approximation scheme via solutions we obtained in Theorem \ref{thm discrete}.

Following the discussion at the beginning of the section, we only need to verify that the conditions in Lemma \ref{lemma 1051} are satisfied.

\begin{lemma}
\label{lemma 1066}
	Let $P_m$ be as given in \eqref{eq 1054}; that is,
	\begin{equation*}
		G_q(P_m,\cdot) = \frac{1}{n+q-1}F_{0,q}(P_m,\cdot)=\frac{1}{n+q-1}\overline{\mu}_m.
	\end{equation*}
	Then there exists $c_0>0$ such that $I_q(P_m)>c_0$ for every $m$.
\end{lemma}
\begin{proof}
	By definition of $G_q(K,\cdot)$ and $I_q(K)$, it follows that
	\begin{equation}
		I_q(P_m) = |G_q(P_m, \cdot)| =\frac{1}{n+q-1}|\overline{\mu}_m|= \frac{1}{n+q-1}|\mu|:=c_0>0.
	\end{equation}
\end{proof}

The rest of the section is devoted to showing that the $P_m$ in \eqref{eq 1054} is uniformly bounded when $\mu$ (not necessarily even) satisfies the subspace mass inequality \eqref{eq 1061}.

For simplicity, we will write
\begin{equation}
	\lambda_i = \frac{i+\min\{i, q-1\}}{n+q-1}.
\end{equation}

For each $\omega\subset \sn$ and $\eta>0$, we define
\begin{equation}
	\mathfrak{N}_{\eta}(\omega) = \{v\in \sn: |v-u|  < \eta, \text{ for some } u\in \omega\}.
\end{equation}

The next lemma shows that when $\mu$ satisfies the subspace mass inequality, then the sequence of approximating discrete measures $\overline{\mu}_m$ satisfies a slightly stronger subspace mass inequality for sufficiently large $m$.
\begin{lemma}
\label{lemma discrete subspace mass inequality}
	Let $\mu$ be a finite Borel measure on $\sn$ and $\overline \mu_m$ be constructed as in \eqref{eq 1059}. If $\mu$ satisfies the subspace mass inequality \eqref{eq 1061}, then there exist $\widetilde{\lambda}_i\in (0, \lambda_i)$, $N_0>0$, and $\eta_0\in (0,1)$ such that for all $m>N_0$,
	\begin{equation}
	\label{eq 6141}
		\frac{\overline{\mu}_m (\mathfrak{N}_{\eta_0}(\xi_i\cap \sn))}{|\mu|}< \widetilde{\lambda_i},
	\end{equation}
	for each $i$-dimensional subspace $\xi_i\subset \rn$ and $i=1,\dots, n-1$.
\end{lemma}
\begin{proof}
	Note that if we can prove the existence of $N_0$, $\eta_0$ for a fixed $i$, then it is simple to find $N_0$ and $\eta_0$ for all $i=1, \dots, n-1$---by taking the maximum of $N_0$ and the minimum of $\eta_0$.
	
	For the rest of the proof, let $i=1,\dots, n-1$ be fixed. We argue by contradiction. If the desired result is false, then there exist sequences $m_j$, $\eta_j$ and $\lambda_i^{(j)}$, and a sequence $\xi^{(j)}$ of $i$-dimensional subspaces such that $m_j\rightarrow \infty$, $\eta_j \rightarrow 0$, $\lambda_{i}^{(j)}\rightarrow \lambda_i$ and
	\begin{equation}
	\label{eq local 1002a}
		\frac{\overline{\mu}_{m_j} (\mathfrak{N}_{\eta_j}(\xi^{(j)}\cap \sn))}{|\mu|}\geq \lambda_i^{(j)}.
	\end{equation}
	Let $e_{1, j},\dots,  e_{i,j}$ be an orthonormal basis of $\xi^{(j)}$. By taking a subsequence, we may assume $e_{k,j}\rightarrow e_k$ for each $1\leq k\leq i$ and that $e_1, \dots, e_i$ are orthonormal. Let $\xi = \Span\{e_1, \dots, e_i\}$. Let $\eta>0$ be an arbitrarily fixed real number. Then, since $\eta_j\rightarrow 0$, we have for sufficiently large $j$,
	\begin{equation}
		\mathfrak{N}_{\eta_j}(\xi^{(j)}\cap \sn)\subset \mathfrak{N}_{\eta}(\xi\cap \sn).
	\end{equation}
	This and \eqref{eq local 1002a} imply
	\begin{equation}
		\frac{\overline{\mu}_{m_j} (\overline{\mathfrak{N}_{\eta}(\xi\cap \sn)})}{|\mu|}\geq \lambda_i^{(j)}.
	\end{equation}
	Now, since $\overline{\mu}_{m_j}$ converges weakly to $\mu$,{ $\overline{\mathfrak{N}_{\eta}(\xi\cap \sn)}$ is compact}, and $\lambda_i^{(j)}\rightarrow\lambda_i$, we have
	\begin{equation}
		\frac{\mu (\overline{\mathfrak{N}_{\eta}(\xi\cap \sn)})}{|\mu|}\geq \lambda_i.
	\end{equation}
	Letting $\eta\rightarrow 0$, we have
	\begin{equation}
		\frac{\mu (\xi\cap \sn)}{|\mu|}\geq \lambda_i,
	\end{equation}
	which contradicts \eqref{eq 1061}.
\end{proof}

For notational simplicity, we will write $\Phi_{\mu}(K,\xi)$ for $\Phi_{0, \mu}(K,\xi)$.

Let $e_1, \dots, e_n$ be an orthonormal basis in $\rn$. We define the following partition of the unit sphere. For each $\delta\in (0,\frac{1}{\sqrt{n}})$, define
\begin{equation}
\label{eq local 20}
	A_{i,\delta} = \{v\in \sn: |v\cdot e_i|\geq \delta, |v\cdot e_j|<\delta, \text{ for } j>i\},
\end{equation}
for each $i=1,\dots, n$. These sets are non-empty since $e_i\in A_{i,\delta}$. They are obviously disjoint. Furthermore, it can be seen that the union of $A_{i,\delta}$ covers $\sn$. Indeed, for any unit vector $v\in \sn$, by the choice of $\delta$, there has to be at least one $i$ such that $|v\cdot e_i|\geq \delta$. Let $i_0$ be the largest $i$ that makes $|v\cdot e_i|\geq \delta$. Then $v\in A_{i_0, \delta}$. We use this spherical partition to prove the following lower bound on $\Phi_{\overline{\mu}_m}(E_m, o)$ when $E_m$ is a sequence of centered ellipsoids.

\begin{lemma}
\label{lemma entropy bound}
	Suppose $1< q< n+1$. Let $\mu$ be a nonzero finite Borel measure on $\sn$ and $\overline{\mu}_m$ be as constructed in \eqref{eq 1059}. Let $E_m$ be a sequence of centered ellipsoids
	\begin{equation}
		E_m = \left\{x\in \rn: \frac{|x\cdot e_{1,m}|^2}{r_{1,m}^2}+ \dots + \frac{|x\cdot e_{n,m}|^2}{r_{n,m}^2}\leq 1\right\},
	\end{equation}
	where $e_{1,m},\dots, e_{n,m}$ is an orthonormal basis in $\rn$ and $0<r_{1,m}\leq \dots\leq r_{n,m}$. Assume further that $e_{1,m}, \dots, e_{n,m}$ converges to an orthonormal basis $e_1,\dots, e_n$ in $\rn$ and $r_{n,m}\geq 1$.
	
	If $\mu$ satisfies the subspace mass inequality \eqref{eq 1061}, then there exists $\delta_0, t_0\in(0,1)$ and $N_0>0$ such that for each $m>N_0$, we have
	\begin{equation}
	\label{eq local 61}
	\begin{aligned}
			\frac{1}{|\overline{\mu}_m|}\Phi_{\overline{\mu}_m}(E_m, o)\geq \log\left(\frac{\delta_0}{2}\right) + t_0 \log r_{n,m}+ (1-t_0)\Bigg[\sum_{i=1}^{\lfloor q\rfloor -1}\frac{2}{n+q-1} &\log r_{i,m} +\frac{q-\lfloor q\rfloor+1}{n+q-1}{\log r_{\lfloor q\rfloor, m}} \\
			+& \sum_{i=\lfloor q\rfloor+1}^{n} \frac{1}{n+q-1} \log r_{i,m}\Bigg].
	\end{aligned}
	\end{equation}
\end{lemma}
Here we adopt the convention that a sum disappears if the upper index is strictly smaller than the lower index. 
\begin{proof}
	Let $A_{i,\delta}$ be constructed as in \eqref{eq local 20} with respect to $e_1, \dots, e_n$.
	
	Since $\mu$ satisfies the subspace mass inequality \eqref{eq 1061}, by Lemma \ref{lemma discrete subspace mass inequality}, there exists $N_0>0$, $\eta_0\in (0,1)$, and $\widetilde{\lambda}_i\in(0,\lambda_i)$ such that for all $m>N_0$, \eqref{eq 6141} holds for each $i$-dimensional proper subspace $\xi_n\subset \rn$. Let $t_0>0$ be sufficiently small so that
	\begin{equation*}
		(1-t_0)\lambda_i>\widetilde{\lambda}_i.
	\end{equation*}
	Hence, for all $m> N_0$, we have
	\begin{equation}
	\label{eq local 21}
		\frac{\overline{\mu}_m (\mathfrak{N}_{\eta_0}(\xi_i\cap \sn))}{|\mu|}< (1-t_0)\lambda_i,
	\end{equation}
	for each $i$-dimensional subspace $\xi_i\subset \rn$ and $i=1,\dots, n-1$. In particular, we let $\xi_i = \Span \{e_1, \dots, e_i\}$.
	
	Observe that for sufficiently small $\delta_0\in (0,1)$, we have
	\begin{equation}
		\bigcup_{j=1}^iA_{j,\delta_0}\subset \mathfrak{N}_{\eta_0}(\xi_i\cap \sn),
	\end{equation}
	and as a consequence of \eqref{eq local 21} and the fact that $A_{j,\delta_0}$ forms a partition of $\sn$, we have
	\begin{equation}
	\label{eq local 60}
		\frac{\sum_{j=1}^i\overline{\mu}_m (A_{j,\delta_0})}{|\overline{\mu}_m|}=\frac{\sum_{j=1}^i\overline{\mu}_m (A_{j,\delta_0})}{|\mu|}< (1-t_0)\lambda_i,
	\end{equation}
	for each $i=1, \dots, n-1$. Here, we also used the fact that $|\overline{\mu}_m|=|\mu|$.
		
	Since $e_{1,m}, \dots, e_{n,m}$ converges to $e_1, \dots, e_n$, there exists $N_1>N_0$ such that for each $m>N_1$,
	\begin{equation}
		|e_{i,m}-e_i|<\frac{\delta_0}{2}, \text{ for } i =1,\dots,n.
	\end{equation}
	Note that since $\pm r_{i,m}e_{i,m} \in E_{m}$, we have for each $v\in A_{i, \delta_0}$
	\begin{equation}
		h_{E_{m}} (v)\geq |v\cdot e_{i,m}|r_{i,m}\geq (|v\cdot e_i|-|v\cdot (e_{i,m}-e_i)|)r_{i,m}\geq \frac{\delta_0}{2}r_{i,m}.
	\end{equation}
	Hence, by the fact that $A_{i,\delta}$ forms a partition of $\sn$, we have
	\begin{equation}
	\label{eq local 24}
	\begin{aligned}
		\frac{1}{|\overline{\mu}_m|}\Phi_{\overline{\mu}_m}(E_m, o) &= \frac{1}{|\overline{\mu}_m|}\sum_{i=1}^n \int_{A_{i,\delta_0}} \log h_{E_m}(v)d\overline{\mu}_m(v)\\
		&\geq \frac{1}{|\overline{\mu}_m|}\sum_{i=1}^n \log\left(\frac{\delta_0}{2}r_{i,m}\right)\overline{\mu}_m(A_{i,\delta_0})\\
		&=\log\left(\frac{\delta_0}{2}\right) + \sum_{i=1}^n\log r_{i,m} \frac{\overline{\mu}_m(A_{i,\delta_0})}{|\overline{\mu}_m|}\\
		&= \log\left(\frac{\delta_0}{2}\right) + \sum_{i=1}^n\log r_{i,m} \cdot \gamma_i,
	\end{aligned}		
	\end{equation}
	where we set
	\begin{equation}
		\gamma_i = \frac{\overline{\mu}_m(A_{i,\delta_0})}{|\overline{\mu}_m|}.
	\end{equation}
	
	We further set $s_i = \gamma_1+\dots+\gamma_i$ for $i=1,\dots, n$ and $s_0=0$. Note that $s_n=1$. We have $\gamma_i = s_i-s_{i-1}$ for $i=1,\dots, n$. Thus,
	\begin{equation}
		\begin{aligned}
			\sum_{i=1}^n\log r_{i,m} \cdot \gamma_i&= \sum_{i=1}^n (s_i-s_{i-1})\log r_{i,m}\\
			&= \log r_{n,m} + \sum_{i=1}^{n-1}s_i(\log r_{i,m}-\log r_{i+1, m}),
		\end{aligned}
	\end{equation}
	where in the last equality, we performed summation by parts. Note that by definition of $s_i$, equation \eqref{eq local 60} simply states
	\begin{equation}
		s_i<(1-t_0)\lambda_i.
	\end{equation}
	This, together with the fact that $r_{i,m}\leq r_{i+1,m}$, implies
	\begin{equation}
	\label{eq local 25}
	\begin{aligned}
		\sum_{i=1}^n\log r_{i,m} \cdot \gamma_i&\geq \log r_{n,m} +\sum_{i=1}^{n-1}(1-t_0)\lambda_i (\log r_{i,m}-\log r_{i+1, m})\\
		&=t_0 \log r_{n,m} + (1-t_0) \left(\sum_{i=1}^{n-1}\lambda_i (\log r_{i,m}-\log r_{i+1, m}) + \log r_{n,m}\right).
	\end{aligned}
	\end{equation}
	
	At this point, we perform summation by parts again and use the definition of $\lambda_i$.
	
	We do it in three cases.
	
	Case 1: $q\in (1,2)$. In this case, we have $\lambda_i = \frac{i+q-1}{n+q-1}$. Thus,
		\begin{equation}
	\label{eq 6156a}
	\begin{aligned}
			&\sum_{i=1}^{n-1}\lambda_i (\log r_{i,m}-\log r_{i+1, m}) + \log r_{n,m}\\
			 = &\lambda_1 \log r_{1,m} + \sum_{i=2}^{n-1}(\lambda_i-\lambda_{i-1})\log r_{i,m} + (1-\lambda_{n-1})\log r_{n,m}\\
			 =   & \frac{q}{n+q-1}\log r_{1,m}  + \sum_{i=2}^{n-1} \frac{1}{n+q-1} \log r_{i,m} + \left(1-\frac{n+q-2}{n+q-1}\right)\log r_{n,m}\\
			 = & \frac{q}{n+q-1}\log r_{1,m}  + \sum_{i=2}^{n} \frac{1}{n+q-1} \log r_{i,m}\\
	\end{aligned}
	\end{equation}

	Case 2: $q\in [2,n)$. Note that if $n=2$, there is no need to consider this case. Hence, for here, we assume $n\geq 3$. We have
	
	\begin{equation}
	\label{eq 6156}
	\begin{aligned}
			&\sum_{i=1}^{n-1}\lambda_i (\log r_{i,m}-\log r_{i+1, m}) + \log r_{n,m}\\
			 = &\lambda_1 \log r_{1,m} + \sum_{i=2}^{n-1}(\lambda_i-\lambda_{i-1})\log r_{i,m} + (1-\lambda_{n-1})\log r_{n,m}\\
			 =   & \frac{2}{n+q-1}\log r_{1,m}  + \sum_{i=2}^{\lfloor q\rfloor -1}\frac{2}{n+q-1} \log r_{i,m} +\frac{q-\lfloor q\rfloor+1}{n+q-1}{\log r_{\lfloor q\rfloor, m}} \\
			 &\phantom{ewqeqwewq}+ \sum_{i=\lfloor q\rfloor+1}^{n} \frac{1}{n+q-1} \log r_{i,m} + \left(1-\lambda_{n-1}-\frac{1}{n+q-1}\right)\log r_{n,m}\\
			 = &  \sum_{i=1}^{\lfloor q\rfloor -1}\frac{2}{n+q-1} \log r_{i,m} +\frac{q-\lfloor q\rfloor+1}{n+q-1}{\log r_{\lfloor q\rfloor, m}} + \sum_{i=\lfloor q\rfloor+1}^{n} \frac{1}{n+q-1} \log r_{i,m},
	\end{aligned}
	\end{equation}
	where in the last equality, we use the fact that $1-\lambda_{n-1}-\frac{1}{n+q-1}=1-\frac{n+q-2}{n+q-1}-\frac{1}{n+q-1}= 0$.
	
	Case 3: $q\in [n,n+1)$. In this case, we have $\lfloor q\rfloor=n$, $\lambda_i = \frac{2i}{n+q-1}$ for $i=1,\dots, n-1$, and
	\begin{equation}
	\label{eq 6157}
		\begin{aligned}
			&\sum_{i=1}^{n-1}\lambda_i (\log r_{i,m}-\log r_{i+1, m}) + \log r_{n,m}\\
			 = &\lambda_1 \log r_{1,m} + \sum_{i=2}^{n-1}(\lambda_i-\lambda_{i-1})\log r_{i,m} + (1-\lambda_{n-1})\log r_{n,m}\\
			 =&   \sum_{i=1}^{n-1}\frac{2}{n+q-1} \log r_{i,m}+(1-\lambda_{n-1})\log r_{n,m}\\
			 =& \sum_{i=1}^{n-1}\frac{2}{n+q-1} \log r_{i,m}+\frac{q-n+1}{q+n-1}\log r_{n,m}\\
			 =& \sum_{i=1}^{n-1}\frac{2}{n+q-1} \log r_{i,m}+\frac{q-\lfloor q\rfloor+1}{q+n-1}\log r_{n,m}.
		\end{aligned}
	\end{equation}
	
	Note that \eqref{eq 6156a},  \eqref{eq 6156} and \eqref{eq 6157} can be written in a uniform way by adopting the convention that a sum disappears if the lower index is strictly bigger than the upper index:
	\begin{equation}
	\label{eq local 26}
	\begin{aligned}
			&\sum_{i=1}^{n-1}\lambda_i (\log r_{i,m}-\log r_{i+1, m}) + \log r_{n,m}\\
			 			 =&  \sum_{i=1}^{\lfloor q\rfloor -1}\frac{2}{n+q-1} \log r_{i,m} + \frac{q-\lfloor q\rfloor+1}{n+q-1}{\log r_{\lfloor q\rfloor, m}} + \sum_{i=\lfloor q\rfloor+1}^{n} \frac{1}{n+q-1} \log r_{i,m}.
	\end{aligned}
	\end{equation}

	Combining \eqref{eq local 24}, \eqref{eq local 25} and \eqref{eq local 26} provides the desired result.
\end{proof}

The following lemma is an estimate on the chord integral of ellipsoids obtained in Xi-LYZ \cite{XLYZ}.
\begin{lemma}[\cite{XLYZ}]
\label{lemma bound on chord integral}
	Suppose $q\in (1, n+1)$ is not an integer. If $E$ is the ellipsoid in $\rn$ given by
	\begin{equation}
		E =\left\{ x\in \rn:\frac{|x\cdot e_1|^2}{r_1^2}+ \dots + \frac{|x\cdot e_n|^2}{r_n^2}\leq 1\right\},
	\end{equation}
	where $e_1,\dots, e_n$ is an orthonormal basis in $\rn$ and $0<r_1\leq \dots\leq r_n$. Then
	\begin{equation}
		\log I_q(E)\leq \Bigg[\sum_{i=1}^{\lfloor q\rfloor -1}2 \log r_{i}+ (q-\lfloor q\rfloor+1){\log r_{\lfloor q\rfloor}} + \sum_{i=\lfloor q\rfloor+1}^{n} \log r_{i}\Bigg]+c(q,n)
	\end{equation}
	where $c(q,n)$ is a constant (not necessarily positive) that only depends on $q$ and $n$.
\end{lemma}

For the rest of the section, we will use symbols like $c(a,b)$ to denote constants that depend only on $a$ and $b$.

We now prove that $P_m$ is uniformly bounded when $q\in (1,n+1)$ is not an integer.
\begin{lemma}
\label{lemma noninteger}
	Suppose $q\in (1,n+1)$ is not an integer. Let $\mu$ be a finite Borel measure on $\sn$ and $\overline\mu_m$ be as constructed in \eqref{eq 1059}. Let $P_m$ be as given in \eqref{eq 1054}. If $\mu$ satisfies the subspace mass inequality \eqref{eq 1061}, then $P_m$ is uniformly bounded.
\end{lemma}
\begin{proof}
Because of homogeneity,  we may assume $\mu$ is a probability measure.

We argue by contradiction and assume that $P_m$ is not uniformly bounded.
	
	Let $E_m$ be the John ellipsoid of $P_m$; that is
	\begin{equation}
	\label{eq local 40}
		E_m\subset P_m\subset n(E_m-o_m)+o_m,
	\end{equation}
	where the ellipsoid $E_m$ centered at $o_m\in \text{int}\, P_m$ is given by
	\begin{equation}
		E_m = \left\{x\in \rn: \frac{|(x-o_m)\cdot e_{1,m}|^2}{r_{1,m}^2}+\dots + \frac{|(x-o_m)\cdot e_{n,m}|^2}{r_{n,m}^2}\leq 1\right\},
	\end{equation}
	for some orthonormal basis $e_{1,m},\dots, e_{n,m}$ in $\rn$ and $0<r_{1,m}\leq \dots\leq r_{n,m}$. Since $P_m$ is not uniformly bounded, by taking a subsequence, we may assume $r_{n,m}\rightarrow \infty$ and $r_{n,m}\geq 1$. By the compactness of $\sn$, we may take a subsequence and assume that $e_{1,m},\dots, e_{n,m}$ converges to $e_1, \dots, e_n$---an orthonormal basis in $\rn$. By the definition of $\Phi_{\overline\mu_m}$, \eqref{eq local 40} and Lemma \ref{lemma entropy bound}, there exists $\delta_0, t_0>0$ and $N_0>0$ such that for each $m> N_0$, we have
	\begin{equation}
	\label{eq local 41}
	\begin{aligned}
		\frac{1}{|\overline\mu_m|}\Phi_{\overline\mu_m}(P_m, o_m)&\geq \frac{1}{|\overline\mu_m|}\Phi_{\mu_m}(E_m, o_m)\\
		&= \frac{1}{|\overline\mu_m|}\Phi_{\mu_m}(E_m-o_m,o)\\
		&\geq \log\left(\frac{\delta_0}{2}\right) + t_0 \log r_{n,m}+ (1-t_0)\Bigg[\sum_{i=1}^{\lfloor q\rfloor -1}\frac{2}{n+q-1} \log r_{i,m} \\&\phantom{wew}+ \frac{q-\lfloor q\rfloor+1}{n+q-1}{\log r_{\lfloor q\rfloor, m}}
			+ \sum_{i=\lfloor q\rfloor+1}^{n} \frac{1}{n+q-1} \log r_{i,m}\Bigg]\\
			& \geq \log\left(\frac{\delta_0}{2}\right) + t_0 \log r_{n,m} + \frac{1-t_0}{n+q-1}\log I_q(E_m)+c(t_0, n,q),
	\end{aligned}
	\end{equation}
	where $c(t_0, n,q)$ is not necessarily positive. Here, the last inequality follows from Lemma \ref{lemma bound on chord integral}. By homogeneity and translation invariance of $I_q$, \eqref{eq local 40}, and the choice of $P_m$, we have
	\begin{equation}
	\label{eq local 80}
		I_q(E_m) = I_q(E_m-o_m)=n^{-\frac{1}{n+q-1}}I_q(n(E_m-o_m)+o_m)\geq n^{-\frac{1}{n+q-1}}I_q(P_m)= n^{-\frac{1}{n+q-1}}\frac{1}{n+q-1}|\overline\mu_m|.
	\end{equation}
	
	Let $y^m \in \mathbb{R}^{\mathcal{N}_m}_+$ be such that $y^m_i = h_{P_m}(v_{i,m})$. By \eqref{eq 791} and \eqref{eq 10410}, we have that $y^m$ is a constant multiple of $z^m$, where $z^m$ is the minimizer to \eqref{eq 6145} with $\xi_{0,\overline{\mu}_m}(z^m)=o$. This, when combined with the homogeneity of $\xi_{0,\overline{\mu}_m}$, implies that $\xi_{0,\overline{\mu}_m}(y^m)=o$. This, \eqref{eq local 80}, that $|\overline\mu_m|=|\mu|$, \eqref{eq local 41}, and that $r_{n,m}\rightarrow \infty$ imply
	\begin{equation}
	\label{eq local 50}
		\Phi_{\overline\mu_m}(P_m, o)\geq\Phi_{\overline\mu_m}(P_m, o_m)\rightarrow \infty, \text{ as } m\rightarrow \infty.
	\end{equation}
	This is a contradiction to Lemma \ref{lemma bound phi}.
\end{proof}

The uniform upper bound for $P_m$ when $q=2,\dots, n$ is an integer is slightly more complicated. We require the following lemma obtained in \cite{XLYZ}, which follows from a simple argument using Jensen's inequality.

\begin{lemma}[\cite{XLYZ}]
\label{lemma monotonicity in q}
	If $K\in \mathcal{K}_o^n$ and $1\leq r<s$, then
	\begin{equation}
		I_r(K)\leq c(r,s) V(K)^{1-\frac{r-1}{s-1}}I_s(K)^\frac{r-1}{s-1},
	\end{equation}
	where $c(r,s)>0$ only depends on $r$ and $s$.
\end{lemma}

The following lemma provides the desired uniform upper bound for $P_m$ when $q$ is an integer.
\begin{lemma}
\label{lemma integer}
	Suppose $q\in\{2,\dots, n\}$. Let $\mu$ be a finite Borel measure on $\sn$ and $\overline\mu_m$ be as constructed in \eqref{eq 1059}. Let $P_m$ be as given in \eqref{eq 1054}. If $\mu$ satisfies the subspace mass inequality \eqref{eq 1061}, then $P_m$ is uniformly bounded.

\end{lemma}
\begin{proof}
	The proof is similar to that of Lemma \ref{lemma noninteger}. Hence, we only outline the necessary changes here.
	
	Using Lemma \ref{lemma discrete subspace mass inequality}, we can conclude the existence of $N_0>0$, $\eta_0\in (0,1)$, and $\widetilde{\lambda}_i\in (0,\lambda_i)$ such that for all $m>N_0$, equation \eqref{eq 6141} holds. We choose $t_0>0$ sufficiently small so that
	\begin{equation}
	\label{eq 6142}
		(1-t_0)\frac{i+\min\{i, q-1\}}{n+q-1}=(1-t_0)\lambda_i>\widetilde{\lambda}_i.
	\end{equation}
	Note that the left-side of \eqref{eq 6142}, when viewed as a function of $q$, is continuous for $q\geq 1$. Therefore, it is possible to choose $q'\in (q,q+1)$ sufficiently close to $q$ so that
	\begin{equation}
	\label{eq 6143}
		(1-t_0)\lambda_i':=(1-t_0)\frac{i+\min\{i, q'-1\}}{n+q'-1}>\widetilde{\lambda}_i,
	\end{equation}
	and
	\begin{equation}
	\label{eq 10251}
	t_0-\frac{1-t_0}{n+q'-1}\frac{q'-q}{q-1}n>0.
\end{equation}
Equations \eqref{eq 6143} and \eqref{eq 6141} now imply that for all $m>N_0$, equation \eqref{eq local 21} holds with $\lambda_i$ replaced by $\lambda_i'$. Thus, Lemma \ref{lemma entropy bound} holds with $q$ replaced by $q'$. Using this in \eqref{eq local 41} and recognizing that $q'$ is now non-integer so that one may once again invoke Lemma \ref{lemma bound on chord integral}, we get
	\begin{equation}
	\label{eq local 70}
		\frac{1}{|\overline\mu_m|}\Phi_{\overline{\mu}_m}(P_m, o_m)\geq \log\left(\frac{\delta_0}{2}\right) + t_0 \log r_{n,m}+\frac{1-t_0}{n+q'-1}\log I_{q'}(E_m)+c(t_0,n,q')
	\end{equation}
	in place of \eqref{eq local 41}. Using Lemma \ref{lemma monotonicity in q} with $r=q$ and $s=q'$, we have
	\begin{equation}
	\label{eq local 71}
		\log I_{q'}(E_m)\geq \frac{q'-1}{q-1} \log I_q(E_m)-\frac{q'-q}{q-1}\log V(E_m)+ c(q,q')
	\end{equation}
	for some constant $c(q,q')$. Combining \eqref{eq local 70} and \eqref{eq local 71}, we have
	\begin{equation}
	\begin{aligned}
		&\frac{1}{|\overline\mu_m|}\Phi_{\overline{\mu}_m}(P_m, o_m)\\
		\geq& \log\left(\frac{\delta_0}{2}\right) + t_0 \log r_{n,m}  +\frac{1-t_0}{n+q'-1}\frac{q'-1}{q-1}\log I_q(E_m)-\frac{1-t_0}{n+q'-1}\frac{q'-q}{q-1}n\log r_{n,m}+c(q,q',n,t_0).
	\end{aligned}
	\end{equation}
Here, we used the fact that $\frac{q'-q}{q-1}>0$ and that $V(E_m)\leq \omega_n r_{n,m}^n$. Therefore,
	\begin{equation}
	\begin{aligned}
		&\frac{1}{|\overline\mu_m|}\Phi_{\overline{\mu}_m}(P_m, o_m)\\
		\geq& \log\left(\frac{\delta_0}{2}\right) + \left(t_0-\frac{1-t_0}{n+q'-1}\frac{q'-q}{q-1}n\right) \log r_{n,m}  + \frac{1-t_0}{n+q'-1}\frac{q'-1}{q-1}\log I_q(E_m) +c(q,q',n,t_0).
	\end{aligned}
	\end{equation}

As argued in \eqref{eq local 80}, the term involving $\log I_q(E_m)$ is bounded from below. Therefore, as $r_{n,m}\rightarrow \infty$, with the help of \eqref{eq 10251}, we may conclude that $\Phi_{\overline{\mu}_m}(P_m, o)\rightarrow \infty$ as in \eqref{eq local 50}. 	This is a contradiction to Lemma \ref{lemma bound phi}.
\end{proof}

\begin{theorem}
	Let $1<q<n+1$. If $\mu$ is a finite Borel measure on $\sn$ that satisfies \eqref{eq 1061}, then there exists a convex body $K\in \mathcal{K}^n$ with $o\in K$ such that
	\begin{equation}
		G_q(K,\cdot)=\mu.
	\end{equation}
\end{theorem}
\begin{proof}
	The result follows immediately from Lemmas \ref{lemma 1051}, \ref{lemma 1066}, \ref{lemma noninteger} (in the case $q$ is a non-integer), and \ref{lemma integer} (in the case $q$ is an integer).
\end{proof}

\textbf{Acknowledgements} The authors are extremely grateful to the referees for their many valuable comments and suggestions. Research of Guo was supported, in part, by NSFC Grants 12126319 and 12126368. Research of Xi was supported, in part, by NSFC Grant 12071277 and  STCSM Grant 20JC1412600. Research of Zhao was supported, in part, by NSF Grant DMS--2132330.

\end{document}